\RequirePackage[l2tabu, orthodox]{nag}

\documentclass[10pt,a4paper,reqno]{amsart}

\usepackage{lmodern}
\usepackage[T1]{fontenc}
\usepackage[utf8]{inputenc}
\usepackage[english]{babel}
\usepackage{microtype} 

\usepackage{amsmath,amsthm,amssymb,latexsym,mathtools,booktabs,enumitem,caption,subcaption,epigraph,graphicx,tikz,hyperref}

\graphicspath{{./}{../hutchinson-figures/}}
\DeclareGraphicsExtensions{.png,.pdf}

\newcommand{\defin}[1]{%
\relax\ifmmode%
\textcolor{blue}{#1}%
\else \textcolor{blue}{\emph{#1}}%
\fi%
}

\newtheorem{theorem}{Theorem}[section]
\newtheorem{proposition}[theorem]{Proposition}
\newtheorem{lemma}[theorem]{Lemma}
\newtheorem{corollary}[theorem]{Corollary}

\newtheorem{problem}[theorem]{Problem}

\theoremstyle{definition}
\newtheorem{definition}[theorem]{Definition}
\newtheorem{example}[theorem]{Example}
\newtheorem{remark}[theorem]{Remark}

\newtheorem{THEO}{Theorem}

\newtheorem{LEMMA}{Lemma}

\newcommand{\Conv}{\mathrm{Conv}}
\newcommand{\diffx}{{\frac{d}{dx}}}
\newcommand{\diff}[1]{{\frac{d^#1}{dx^#1}}}

\newcommand{\bC}{\mathbb{C}}
\newcommand{\bR}{\mathbb{R}}

\newcommand{\bZ}{\mathbb{Z}}

\renewcommand{\Im}{{\mathrm{Im}}}
\renewcommand{\Re}{{\mathrm{Re}}}

\newcommand{\C}{\mathcal{C}}

\newcommand{\al}{\alpha}
\newcommand{\La}{\Lambda}

\newcommand{\la}{\lambda}
\newcommand {\eps}{\epsilon}

\newcommand{\invset}[1]{{{\mathcal I}_{#1}^T}}
\newcommand{\minvset}[1]{{{\mathrm M}_{#1}^T}}

\newcommand{\hinvset}[1]{{{\mathcal{H}}_{#1}^T}}
\newcommand{\hminvset}[1]{{{\mathrm{HM}}_{#1}^T}}

\newcommand{\ctinvset}[1]{{{\mathcal{CH}}_{#1}^T}}
\newcommand{\ctminvset}[1]{{{\mathrm{CHM}}_{#1}^T}}

\newcommand{\cttminvset}[1]{{{\mathrm{C_2HM}}_{#1}^T}}

\newcommand{\disk}{D}
\newcommand{\cdisk}{{\bar{D}}}
\newcommand{\thsup}{\mathrm{th}}
\newcommand {\NE}{\mathbf{NE}}

\newcommand{\fidx}{\rho}

\numberwithin{equation}{section}

\title[An inverse problem in P\'olya--Schur theory. I.]
{An inverse problem in P\'olya--Schur theory. I. Non-degenerate and degenerate operators}

\author[P. Alexandersson]{Per Alexandersson}
\address{Department of Mathematics, Stockholm University,
S-10691, Stockholm, Sweden}
\email{per.w.alexandersson@gmail.com}

\author[P. Br\"and\'en]{Petter Br\"and\'en}
\address{Department of Mathematics,
Royal Institute of Technology,
SE-100 44, Stockholm, Sweden}
\email{pbranden@kth.se}

\author[B. Shapiro]{Boris Shapiro}
\address{Department of Mathematics,
Stockholm University,
S-10691, Stockholm, Sweden}
\email{shapiro@math.su.se}

\begin{document}

\begin{abstract}
Given a linear ordinary differential operator $T$ with polynomial coefficients,
we  study  the class of closed subsets of the complex plane such that $T$ sends any polynomial 
(resp. any polynomial of degree exceeding a given positive integer) with all 
roots in a given subset to a polynomial with 
all roots in the same subset or to $0$. 
Below we discuss some general properties of such invariant subsets as well 
as the problem of existence of the minimal under inclusion invariant subset.  
\end{abstract}

\maketitle

{\emph{\small 
If a new result is to have any value,
it must unite elements long since known, \\ 
but till then scattered and seemingly foreign to each other, 
and suddenly introduce\\ order  where the appearance of disorder reigned. 
Then it enables us to see at a glance\\ each of these elements 
at a place it occupies in the whole.}

{\hskip6cm \small  --- H.~Poincar\'e, Science and Hypothesis}

\tableofcontents

\bigskip
\section{Introduction}

In 1914, generalizing some earlier results of E.~Laguerre, G.~P\'olya and I.~Schur \cite{PolyaSchur1914}
created a new branch of mathematics now referred to as the P\'olya--Schur theory.
The main result of \cite{PolyaSchur1914} is a  complete characterization of linear operators
acting diagonally in the monomial basis of $\bR[x]$ and  sending any polynomial with
all real roots to a polynomial with all real roots (or to $0$).
Without the requirement of diagonality of the  action a characterization of such
linear operators was obtained   by the
second author jointly with late J.~Borcea \cite{BorceaBranden2009}.

\medskip
The main question  considered in the P\'olya--Schur theory \cite{CravenCsordas2004}
can be formulated as follows.

\begin{problem}\label{prob1} 
Given a subset $S\subset \bC$ of the complex plane, describe the
semigroup
of  all linear operators $T:\bC[z]\to\bC[z]$
sending any polynomial with roots in $S$ to a
polynomial with roots in $S$ (or to $0$).
\end{problem}

\begin{definition}\label{def0}
If an operator $T$ has the latter property, then we say that
\defin{$S$ is a $T$-invariant set}, or that \defin{$T$ preserves $S$}.
\end{definition}

So far Problem~\ref{prob1} has only been solved for the circular domains (i.e., images of
the unit disk under M\"obius transformations), their boundaries \cite{BorceaBranden2009}, 
and  more recently for strips \cite{BrandenChasse2017}.
Even a very similar case of the unit interval is still open at present. 
It seems that for a somewhat general class of subsets $S\subset \bC$, Problem~\ref{prob1}
is  out of reach of all currently existing methods.  

\medskip
In this paper, we consider
an inverse problem in the P\'olya--Schur theory which seems both natural and more 
accessible than Problem~\ref{prob1}. We will restrict ourselves to 
consideration of \emph{closed} $T$-invariant subsets.

\begin{problem}\label{prob:main}
Given a linear operator $T:\bC[x]\to\bC[x]$, characterize all closed $T$-invariant subsets of the complex plane. Alternatively, find a sufficiently large class of 
$T$-invariant sets.
\end{problem} 

For example, if $T = \diff{j}$, then a closed subset $S\subseteq \bC$ is $T$-invariant if and only if it is convex. 
Although it seems too optimistic  to hope for a complete  solution of
Problem~\ref{prob:main} for an arbitrary linear operator $T$,
we present below a number of relevant results  valid
for  linear ordinary differential operators of finite order. 
(Note that  an arbitrary linear operator $T:\bC[x]\to\bC[x]$ can be represented as a formal
linear differential operator with polynomial coefficients,
i.e., $T=\sum_{j=0}^\infty Q_j(x)\diff{j}$ where each $Q_j(x)$ is a polynomial, see \cite{Peetre1959}). 
To move further, we need to introduce some basic notions. 

\begin{definition}\label{def:degNonDegFuchs}
Given a linear ordinary differential operator 
\begin{equation}\label{eq:main}
T = \sum_{j=0}^k Q_j(x)\diff{j}
\end{equation} 
of  finite order $k\ge 1$  with polynomial coefficients, define its \defin{Fuchs index} as 
\[
\fidx_T = \max_{0\le j\le k} (\deg(Q_{j})-j).
\]
Alternatively, the Fuchs index can be defined as the maximal difference between the output and input polynomial,
when acted upon by $T$:
\[
\fidx_T = \max_{p \in \bC[x]} \left( \deg( T(p) ) - \deg(p) \right).
\]
An operator $T$ is called \defin{non-degenerate} if $\fidx_T = \deg(Q_k)-k$, and \defin{degenerate} otherwise.
In other words, $T$ is non-degenerate if $\fidx_T$ is realized by the leading coefficient of $T$. 
We say that $T$ is \defin{exactly solvable} if its Fuchs index is zero.
\end{definition}

A few operators illustrating the situation are shown in Table~\ref{tab:operatorExamples},
with some of their properties listed.

\begin{table}[!ht]
 \centering
\begin{tabular}{lcl}
\textbf{Operator} & \textbf{Fuchs index} & \textbf{Properties} \\ 
\midrule 
$(x^3+2x) \diff{3} + x \diff{2} + 1 $  & 0 & Exactly solvable, non-degenerate \\[0.3cm]
$ (x+1)\diff{3} + x^4 \diff{2} + 2x $  & 2 & Degenerate \\[0.3cm]
$ x^2\diff{3} + 4\diff{2}$  & -1 & Non-degenerate \\[0.3cm]
\bottomrule
\end{tabular}
\vskip0.3cm
\caption{Three examples of differential operators. }\label{tab:operatorExamples}
\end{table}

\begin{definition}\label{def:T_n}
Given a linear operator $T:\bC[x]\to \bC[x]$, we denote by $\invset{n}$ 
the collection  of all closed  subsets $S\subset \bC$ 
such that for every polynomial of degree $n$ with roots in $S$,
its image  $T(p)$ is either $0$ or has all roots in $S$. 
In this situation, we say that $S$ \defin{belongs to the class $\invset{n}$} or, equivalently,  that $S$ is \defin{$T_n$-invariant}. 

Similarly, a \emph{closed} set $S$  belongs to  \emph{the class $\invset{\geq n}$}
if for every polynomial of degree \emph{at least} $n$ with roots in $S$,
its image  $T(p)$ is either $0$ or has all roots in $S$. In this case we say that $S$ is \defin{$T_{\ge n}$-invariant}.  By definition, the class $\invset{\geq 0}$ 
coincides with the class of all $T$-invariant sets.  We say that a set $S \in \invset{n}$ 
(resp. $S \in \invset{\geq n}$) is \defin{minimal}
if there is no closed proper nonempty subset of $S$ belonging to  $\invset{n}$ 
(resp. to $\invset{\geq n}$).
\end{definition}

\begin{remark}
Obviously, for any $T$ and 
any $n$, the whole complex plane $\bC$ is  a trivial example of  a set belonging 
to both $\invset{n}$ and $\invset{\geq n}$. On the one hand,  in case when the operator $T$ preserves the space of polynomials 
of degree $n$ it is more natural to study   the 
class $\invset{n}$. In particular, any exactly solvable operator preserves the degree of 
polynomials it acts upon (except for possibly finitely many exceptions in low degrees). 
Thus, for an exactly solvable operator, it makes sense to consider the class $\invset{n}$ and its elements for all (sufficiently large) $n$ and study their behavior when $n\to\infty$. 
On the other hand, for an arbitrary linear operator $T$ it is more natural to consider non-trivial subsets of $\bC$ 
belonging to $\invset{\ge n}$  where $n$ is any non-negative integer.   
 Observe that  families of sets belonging to  $\invset{n}$ (resp. $\invset{\geq n}$)   
are closed under  taking the intersection. 
\end{remark}

\medskip
In the present paper (which is the first part of two) 
we study  the class $\invset{\ge n}$ for an arbitrary $T$ of the form \eqref{eq:main}. 
The sequel article \cite {AlBrSh2} is devoted to the study of the class  
$\invset{n}$ and also of the so-called Hutchinson invariant sets for exactly solvable operators and their relation to the classical complex dynamics. A recent paper \cite{AHNST} contains the results of the first and the third authors jointly with N.~Hemmingsson, D.~Novikov, and  G.~Tahar on a similar topic where we provide many details about  the so-called continuous Hutchinson invariant sets for operators $T$ of order $1$.  

\medskip
The structure of the paper is as follows. In Section~\ref{sec:general},
we present and prove some general results about  $\invset{\ge n}$
for an arbitrary operator $T$ (with non-constant leading term).
In Section~\ref{sec:non-degenerate}, we prove all results related to non-degenerate operators.
In Section~\ref{sec:degenerate} and Section~\ref{sec:deg}, we prove all results related to degenerate differential operators
including operators with constant leading term.   In Section~\ref{sec:trop}  we provide preliminary information about the asymptotic root behavior for bivariate polynomials used
in Section~\ref{sec:deg}.
In Section~\ref{sec:examples} we discuss several natural set-ups and problem formulations similar to that of the current paper.
Finally, Section~\ref{sec:final} contains a number of open problems connected to the presented results.

\medskip 
\noindent 
\textbf{Acknowledgements.} 
Research of the third author was supported by the grants VR 2016-04416 and VR 2021-04900 of the Swedish Research Council. He wants to thank Beijing Institute for Mathematical Sciences and Applications (BIMSA) for the hospitality in Fall 2023.

\section{General properties of invariant sets}\label{sec:general}


\begin{definition} 
Given an operator  $T$ of the form \eqref{eq:main} with $Q_k(x)$ different from a constant, 
denote by $\Conv (Q_k)\subset \bC$ the convex hull of the zero locus of  $Q_k(x)$. 
We refer to $\Conv (Q_k)$ as the \defin{fundamental polygon} of $T$.
\end{definition}

The next proposition contains  basic information about invariant sets in $\invset{\geq n}$.

\begin{theorem}\label{th:generalGeN} 
The following facts hold: 

\begin{enumerate}[label=\normalfont{(\arabic*)}]

\item for any operator $T$ as in \eqref{eq:main} 
and any non-negative integer $n$, every $S \in \invset{\geq n}$ is convex;

\item for any operator $T$ as in \eqref{eq:main} and any non-negative integer $n$, 
if $S$ is an unbounded closed set belonging to $\invset{\geq n}$, then $S$ is $T$-invariant, i.e., $S$ belongs to $\invset{\geq 0}$;

\item for any $T$ as in \eqref{eq:main} with $Q_k(x)$ 
different from a constant and any non-negative integer $n$, every $S \in \invset{\geq n}$ contains  the fundamental polygon $\Conv(Q_k)$; 

\item for any $T$ as in \eqref{eq:main} 
with $Q_k(x)$ different from a constant and any non-negative integer $n$, 
the set $\invset{\geq n}$ has a unique minimal (under inclusion) element.  
\end{enumerate}
\end{theorem}
\begin{proof}
{Item \upshape{(1)}}. Fix $S \in \invset{\geq n}$ and choose $x_1,x_2\in S$. 
Take $p(x)=(x-x_1)^{m}(x-x_2)^{m}$ for sufficiently large $m$, and consider $p^{(\ell)}(x)$.
Then 
\[
p^{(\ell)}(x) = \sum_{j=0}^\ell \binom{\ell}{j} \frac{m!}{(m-j)!} \frac{m!}{(m+j-n)!}(x-x_1)^{m - j} (x-x_2)^{m +j-\ell}
\]
which implies  that
\begin{align*}
q(x)\coloneqq \frac{p^{(\ell)}(x)}{  (x-x_1)^{m-\ell} (x-x_2)^{m-\ell}  } &=
\sum_{j=0}^\ell \binom{\ell}{j} (m)_j(m)_{\ell-j} (x-x_1)^{\ell-j} (x-x_2)^{j}.
\end{align*}
Dividing both sides by $m^\ell$ and expanding the Pochhammer symbols, we see that
\begin{align*}
m^{-\ell}q(x) &=  \left( \sum_{j=0}^\ell \binom{\ell}{j} (x-x_1)^{\ell-j} (x-x_2)^{j} \right) + \frac{R_1(x)}{m} +  \frac{R_2(x)}{m^2} + \dotsb \\
 &=  \left((x-x_1)+(x-x_2)\right)^\ell + O(m^{-1}) R(x). 
\end{align*}

Using the latter expression, we obtain
\[
p^{(\ell)} = m^\ell ((x-x_1)(x-x_2))^{m-\ell} \left(\left(2x-x_1-x_2\right)^\ell + O(m^{-1}) R(x) \right).
\]
Therefore,
\begin{align*}
\frac{T(p(x))}{m^\ell } &= Q_k(x) ((x-x_1)(x-x_2))^{m-k} \left( \left(2x-x_1-x_2\right)^\ell + O(m^{-1}) R(x) \right) \\
 & + \sum_{j=1}^{k} \frac{ Q_{k-j} ((x-x_1)(x-x_2))^{m-k+j} }{m^j}  \left(\left(2x-x_1-x_2\right)^\ell + O(m^{-1}) R_j(x) \right).
\end{align*}
All terms in the above sum approaches $0$ as $m$ gets large, implying that 
the roots of $T(p(x))$ are close to that of
\[
Q_k(x)((x-x_1)(x-x_2))^{m-\ell} \left(2x-x_1-x_2)\right)^\ell.
\]
 Since $\frac{x_1+x_2}{2}$  is a  root of the latter polynomial, the original set $S$ is convex.
 
 \medskip
 \noindent
 {Item \upshape{(2)}}. Assume that $S$ is an unbounded set belonging to $\invset{\ge n}$ for some positive $n$. 
 Take some polynomial $p$ of degree less that $n$ with roots in $S$. 
 Consider a $1$-parameter family of polynomials of degree $n$ of the form 
 $P_t\coloneqq (x-\al(t))^{n-\deg p}p(x)$, $t\in [0,+\infty)$,  
 where $\al(t)$ is a variable point in $S$ which continuously depends on $t$ and 
 escapes to $\infty$ when $t\to +\infty$. (Such a family obviously exists since $S$ is convex and unbounded.) 
 Consider the polynomial family $T(P_t)$. 
 Since $S\in \invset{\ge n}$, the roots of $T(P_t)$ belong to $S$ for any finite $t$ and 
 continuously depend on $t$. Since $S$ is closed the same holds for the limit of the 
 roots of $T(P_t)$ which do not escape to infinity. 
 Notice that the set of finite limiting roots exactly coincides 
 with the set of roots of $T(p)$ which finishes the proof of item~\upshape{(ii)}. 

 \medskip
 \noindent
 {Item \upshape{(3)}}. Take an   arbitrary $T$ with $Q_k(x)$ different 
from a constant, any non-negative integer $n$,  and an arbitrary set $S \in \invset{\geq n}$. 
Set $p(x) = (x-\alpha)^m$, where $\alpha \in S$.
Then
\[
\frac{T(p(x))}{(m)_k}  = \sum_{j=0}^k Q_j(x) \frac{(m)_j}{(m)_k} (x-\alpha)^{m-j}.
\]
If  $m \rightarrow \infty,$ then $\frac{(m)_j}{(m)_k} \rightarrow 0$ for $j<k$.
Hence, the roots of $T(p(x))$ approach those of $Q_k(x)(x-\alpha)^{m-k}$
as $m$ grows.

 \medskip
 \noindent
 {Item \upshape{(4)}}. Observe that for any differential operator $T$ as above, 
the set $\invset{\geq n}$ is non-empty since it at least contains  the whole $\bC$. 
Now notice that by  items~\upshape{(1)}~--~\upshape{(2)}, the intersection of all sets 
in $\invset{\geq n}$ is non-empty. Indeed  each of them contains all 
roots of $Q_k(x)$.  Since this intersection is convex it also contains  the convex hull $\Conv(Q_k)$ of the roots of $Q_k(x)$. 
Since $\invset{\geq n}$ consists of closed convex sets with a non-empty common intersection, 
there is the unique minimal set in $\invset{\geq n}$.
\end{proof}

Let us denote by $\minvset{\geq n}$ the unique minimal element in $\invset{\geq n}$ whose existence is guaranteed by item (4) of Theorem~\ref{th:generalGeN} .
The following consequence of Theorem~\ref{th:generalGeN} is straightforward.

\begin{corollary}\label{cor:natural} 
{\normalfont(i)} Under the assumption that $Q_k(x)$ is not constant, one has the sequence of inclusions of closed convex sets
\begin{equation}\label{eq:limit}
\minvset{\ge 0} \supseteq \minvset{\ge 1} \supseteq \dotsb.
\end{equation}

\noindent
{\normalfont(ii)} Under the same assumption, if for some $n$, there exists a compact set $S\in \invset{\ge n}$ 
then $\minvset{\ge m}$ is compact for all $m\ge n$ and there exists a well-defined limit 
\begin{equation}\label{eq:limit2}
\defin{ \minvset{\infty} } \coloneqq \lim_{n\to\infty}  \minvset{\ge n}. 
\end{equation}
Obviously, $\minvset{\infty}$ is a closed convex compact set. 
\end{corollary}

\begin{remark}
The assumption that $Q_k(x)$ is different from a constant is important for the existence of 
the unique minimal under inclusion element in $\invset{\geq n}$.  
Many operators with a constant leading term violate this property.
For example,  for $T = \diffx,$ every convex closed
subset of $\bC$ belongs to $\invset{\geq n}$
for every non-negative integer $n$. In fact, every point in $\bC$ is a minimal set for $T=\diffx$. 
 More details about operators with a constant leading term can be found in Section~\ref{sec:deg}.
\end{remark}

\begin{remark}
Corollary~\ref{cor:nonDegenerateStableDisks} of the next section  implies that for a non-degenerate $T$,
the minimal set  $\minvset{\geq n}$ is compact for any sufficiently large $n$. However this compactness property might fail for small $n$. Theorem~\ref{thm:limit} below claims that
$\minvset{\infty}$ coincides with the fundamental polygon $Conv(Q_k)$. 

On the other hand,  as we will show  in Proposition~\ref{pr:degen} of \S~\ref{sec:degenerate}, for any degenerate operator $T$ and non-negative integer $n$, every set in $\invset{\ge n}$ and, in particular, $\minvset{\ge n}$ is unbounded implying that  compact invariant sets  exist  if and only if $T$ is non-degenerate operators only.   Together with item (2) of Theorem~\ref{th:generalGeN} this implies that for any degenerate $T$ and any positive integer $n$,  $\invset{\geq n}=\invset{\geq 0}$ and if either at least one of $Q_k(x)$ or $Q_0(x)$ has positive degrees then 
\begin{equation}\label{eq:limitdeg}
\minvset{\ge 0} = \minvset{\ge 1} =  \minvset{\ge 2}\dotsb
\end{equation}
which is a very essential difference between the cases of non-degenerate and degenerate operators. (The fact that every $T$-invariant set $S$ contains all the roots of $Q_0(x)$ follows from the trivial identity $T(1)=Q_0(x)$ and that $S$ contains all roots of $Q_k(x)$ is shown in item (3) of Theorem~\ref{th:generalGeN}). 
\end{remark}

\section{Non-degenerate operators}\label{sec:non-degenerate}

The  main result of this section is Corollary~\ref{cor:nonDegenerateStableDisks},
claiming that for a fixed non-denegerate differential operator $T$,
there exists a nonnegative integer $n$ such that $\invset{\geq n}$ contains all sufficiently large disks. This implies compactness of the minimal set  $\minvset{\geq n}$ for large $n$. 
 Unfortunately, at present we do have an explicit  description the boundary of  $\minvset{\geq n}$ for a given $T$ and $n$. Our best result in this direction is  Theorem~\ref{thm:limit} which claims that the limit $\minvset{\infty}$ coincides with the fundamental polygon $Conv(Q_k)$.
 \medskip

The next example shows that Corollary~\ref{cor:nonDegenerateStableDisks} is the best we can hope for,
as there exist non-degenerate exactly solvable operators for which $\minvset{\geq n}$ is non-compact for small values of $n$.
\begin{example}
Consider the non-degenerate exactly solvable operator given by 
\begin{equation}
T = \left(-\frac{x^2}{4} + \frac{x}{4}\right)\frac{d}{dx^2} + \left( \frac{x}{4} - \frac{1}{2}\right)
\diffx +  1.
\end{equation}
We have chosen $T$ in such a way that for every $z \in \bC$,
\begin{equation} \label{eq:exactlySovlableNonCompact}
 T\left[ (x-z)^2 \right] = \left(x - 2 z \right) \left( x - (\tfrac{z}{2} + \tfrac{1}{2}) \right).
\end{equation}
Take any closed subset $S \in \invset{\geq 2}$. 
The first factor in \eqref{eq:exactlySovlableNonCompact} ensures that if $z \in S$, then we also have $2z \in S$.
The second factor ensures that if $z \in S$, then $\frac12(z+1) \in S$.
These two facts imply that $S$ must contain the interval $[1,\infty)$ of the real axis.
In particular, the minimal set $\minvset{\geq 2}$ cannot be bounded.

Moreover, the image of $(x-1)^4$ has $-3$ as root. This then implies that 
the entire real line lies in $\minvset{\geq 2}$.
Finally, the image of $(x + 1)^2 (x - 1)^2$ has two complex (conjugate) roots,
and this then implies that $\minvset{\geq 2}$ is in fact the entire $\bC$.
\end{example}

\subsection{Existence of invariant disks} 

In this subsection we will show that for any non-degenerate operator $T$, the collection $\invset{\geq n}$ of its $n$-invariant sets contains large disks centered at $0$ for all sufficiently large $n$. 

Define the \defin{$n^\thsup$ Fuchs index} of a linear operator $T : \bC_n[x] \rightarrow \bC[x]$ as given by 
\begin{equation}
\rho = \rho_n = \max_{0\leq j \leq n} (\deg T(x^j)-j), 
\end{equation}
and call $T$ \defin{non-degenerate} if $\deg T(x^n)-n = \rho_n$. 
Set $G_T(x,y) \coloneqq T[(1+xy)^n]$ and note that 
there exist polynomials $P^n_\ell$, $\ell = -n, \dotsc, \rho$, of degree at most $n$, 
such that
\begin{equation}
G_T(x,y)= \sum_{-n \leq \ell \leq \rho} x^\ell P_\ell^n(xy).
\end{equation}
Thus $T$ is non-degenerate if and only if the degree of $P_\rho^n$ is $n$. 
If $T = \sum_{j=0}^k Q_j(x)\diff{j}$ is   a differential operator of order $k$, then
\begin{equation}\label{rn}
G_T(x,y) = \sum_{j=0}^k j! x^{-j} Q_j(x) \binom{n}{j} (xy)^j(1+xy)^{n-j},
\end{equation}
and it follows that
\begin{equation}\label{prn}
P_\ell^n(x)= \sum_{j=0}^k j!a_{\ell,j}\binom{n}{j} x^j (1+x)^{n-j},
\end{equation}
where $a_{\ell,j}$ is the coefficient of $x^{j+\ell}$ in $Q_j(x)$. Define 
\begin{equation}\label{fprn}
f_\ell^n(x)= \sum_{j=0}^k j!a_{\ell,j}\binom{n}{j}x^j.
\end{equation}

\medskip 

In what follows, $\disk_R$ denotes the open disk $\{x \in \bC : |x|< R\}$,
and $\cdisk_R$ is the closure of $\disk_R$.
We also define $\Omega_R$ as the open set $\{(x,y) \in \bC^2 : |x| > R \text{ and } |y| > 1/R\}$.

\begin{proposition}[{\cite[Thm.7]{BorceaBranden2009}}]\label{prop:bb09}
Let $T : \bC_n[x] \rightarrow \bC[x]$ be a linear operator of rank greater than one. The disk $\cdisk_R$ is $T_n$-invariant if and only if $G_T(x,y) \neq 0$ for 
all $(x,y) \in \Omega_R$. 
\end{proposition}

\begin{theorem}\label{genethm}
Suppose $T : \bC_n[x] \rightarrow \bC[x]$ is a non-degenerate linear operator with $n^\thsup$ Fuchs 
index $\rho$. Let $g(x)$ be the greatest common divisor of $\{P_\ell^n(x)\}_\ell$. 
Then the closed disk $\cdisk_R = \{x : |x| \leq R\}$ is $T_n$-invariant 
for all sufficiently large $R>0$ if 
and only if 
\begin{enumerate}
\item all zeros of $g(x)$ lie in $\{x: |x| \leq 1\}$;
\item all zeros of $P_\rho^n(x)/g(x)$ lie in  $\{x : |x|< 1\}$. 
\end{enumerate}
\end{theorem}
\begin{proof}
Suppose $T : \bC_n[x] \rightarrow \bC[x]$ is a non-degenerate linear operator. We first prove that conditions (1) and (2) are sufficient for $T_n$-invariance. Indeed assume that (1) and (2) hold.   
Since  $\deg P_\ell^n \leq \deg P_\rho^n=n$ for all $j$, and the zeros of $P^n_\rho(x)/g(x)$ lie in the 
open unit disk, there is a positive constant $C$ such that $|P_\ell^n(x)/P_\rho^n(x)| < C$ 
for all $|x| \geq 1$ and all $\ell$. Hence, for sufficiently large $R$,  if $(x,y) \in \Omega_R$, then 
\begin{equation}\label{tratt}
\left| \frac {G_T(x,y)} { x^{\rho} P_\rho^n(xy)}-1\right|  =\left| \sum_{\ell=-n}^{\rho-1} x^{\ell-\rho} \frac {P_\ell^n(xy)}{P_\rho^n(xy)}  \right| \leq \sum_{\ell =-n}^{\rho-1} R^{-(\rho -\ell)} C < \frac C {R-1}<1.
\end{equation}
For such $R$, the disk $\cdisk_R = \{x : |x| \leq R\}$ is $T_n$-invariant by Proposition~\ref{prop:bb09}.

Suppose $\cdisk_R = \{x : |x| \leq R\}$ is $n$-invariant for $R$ sufficiently large. 
If $g(x)$ has a zero in $\{x : |x|>1\}$, then $G_T(x,y) = 0$ for 
some $(x,y) \in \Omega_R$, and by Proposition~\ref{prop:bb09}, the disk $\cdisk_R$ is not $T_n$-invariant.
To get a contradiction, suppose $(P_\rho^n/g)(y)=0$, where $|y|\geq 1$. 
Consider a sequence $\{y_j\}_{j=1}^\infty$, where $P_\rho^n(y_j) \neq 0$, $|y_j|>1$, 
and $\lim_{j \to \infty} y_j = y$. 
Let 
\[
B_j(x) \coloneqq x^nG_T(x,y_j/x)/g(y_j) = \sum_{\ell \leq \rho} x^{\ell+n} (P_\ell^n/g)(y_j) .
\]
Since $(P_\rho^n/g)(y)=0$ and $P_\ell^n(y)\neq 0$ for some $\ell$, we see that at least one zero, say $x_j$, of $B_j(x)$ 
tends to $\infty$ as $j \to \infty$. Hence for  
\[
R_j \coloneqq \frac12 |x_j|\left(1+\frac 1 {|y_j|}\right),
\]
 $(x_j, y_j/x_j) \in \Omega_{R_j}$, while $G_T(x_j,y_j/x_j)=0$. By Proposition~\ref{prop:bb09}, $D_{R_j}$ is not $T_n$-invariant for any  $R_j$.
\end{proof}

\medskip 

Recall that the M\"obius map $x \mapsto \frac{x}{1+x}$ sends 
the set $\{x \in \bC : \Re(x)\geq -{1}/{2} \}$ 
to the unit disk.

\begin{theorem}\label{th:15} 
For $T=\sum_{j=0}^k Q_j(x) \frac{d^j}{dx^j}$ and $n\in \mathbb{N}$, assume that $T$ is non-degenerate  as a linear operator $T : \bC_n[x] \rightarrow \bC[x]$. Let $\rho$ be the $n^\thsup$ Fuchs index of $T$, 
and $a_{\rho,j}$ be the coefficient of $x^{\rho+j}$ in $Q_j$. Then the closed disk $\cdisk_R$ is $T_n$-invariant for all sufficiently large $R$ if and only if 
\begin{enumerate}
\item all zeros of the polynomial $h:=\gcd(f_{-n}^n, f_{-n+1}^n, \ldots, f_{\rho}^n)$ have real part greater than or equal  to $-1/2$. Equivalently, there is no $\beta$ with $\Re(\beta)>1/2$, such that 
\[
\sum_{j=0}^k x^{-j}j!\binom n j Q_j(x) \beta^j \equiv 0,
\]
and,
\item all zeros of the polynomial $f^n_\rho/h$ have real part greater than $-{1}/{2}$.
\end{enumerate}
\end{theorem}
\begin{proof}
We want to translate conditions (1) and (2) of Theorem~\ref{genethm} into this setting. This is done by \eqref{rn}, \eqref{prn}, \eqref{fprn} and suitable M\"obius transformations.


\end{proof}

\begin{example}\label{exampleimp}
Let $T= Q_1(x)\frac{d}{dx} + Q_0(x)$ be a non-degenerate linear operator of order $1$. Suppose first that 
\begin{equation}\label{klo}
Q_0(x)+n\beta x^{-1}Q_1(x) \equiv 0
\end{equation}
for some $\beta$. Then 
\[
T=P\cdot (\beta n - xD),
\]
where $P=P(x)$ is some polynomial. But then 
\[
T(x-R)^n= nP\cdot (x-R)^{n-1} ( (\beta-1)x-\beta R),
\]
has a zero outside $\{ x : |x| \leq R\}$ if and only if $|\beta/(\beta-1)|>1$ which is equivalent to $\Re\, \beta >1/2$. This explains condition (1) in Theorem~\ref{th:15}.   

Suppose \eqref{klo} is not satisfied for any $\beta$. If $\deg Q_1 > \deg Q_0+1$, then (2) is always satisfied. Suppose $\deg Q_1 = \deg Q_0+1$, and let $a_i$ be the leading coefficient of $Q_i$. The polynomial in (2) equals 
$
a_0+na_1x.
$ 
Hence condition (2) is equivalent to  $\Re\left({a_0}/{a_1}\right)<n/2$. 
\end{example} 

\begin{proposition}
Let $T : \bC_n[x] \rightarrow \bC_n[x]$ be a diagonal operator, i.e., 
\[
T(x^i) = \lambda_i x^i, \quad 0\leq i \leq n.
\]
The following conditions are equivalent: 
\begin{enumerate}
\item There is a compact non-empty $T_n$-invariant set $K\neq \{0\}$,
\item $\overline{D}_1$ is $T_n$-invariant,
\item $\overline{D}_R$ is $T_n$-invariant for all $R>0$,
\item all zeros of 
the polynomial 
\[
\sum_{i=0}^n \lambda_i \binom n i x^i
\]
lie in $\overline{D}_1$.
\end{enumerate}
\end{proposition}

\begin{proof}
 Since the symbol of $T$ is given by 
\[
G_T(x,y)=T[(1+xy)^n] = \sum_{i=0}^n \lambda_i \binom n i (xy)^i,
\]
we see that the disk $\overline{D}_R$ is $T_n$-invariant if and only if all zeros of 
the polynomial 
\[
\sum_{i=0}^n \lambda_i \binom n i x^i
\]
lie in $\overline{D}_1$. This proves the equivalence of (2), (3) and (4). 

Suppose that (1) holds for some $K$, but not (2). Let $\zeta \in K$ be of maximal modulus. Since, the polynomial in (3) has zero outside the unit disk, the polynomial 
\[
T((x-\zeta)^n)= (-\zeta)^n\sum_{i=0}^n \lambda_i \binom n i \left(-\frac x \zeta\right)^i 
\]
has a zero outside $K$, a contradiction. 
\end{proof}

\begin{corollary}\label{cor:nonDegenerateStableDisks}
If $T$ is a non-degenerate differential operator, then there is an integer $N_0$ and a 
positive number $R_0$ such that the disk $D_R:=D(0,R)$ is $T_n$-invariant whenever $n \geq N_0$ and $R \geq R_0$. 
\end{corollary}
\begin{proof}
Note that the zeros of $f_n^T$ approach $0$ as $n \to \infty$. 
Since $a_{\rho,k} \neq 0$, we see by \eqref{prn} that there exist positive numbers  $N_0$ and $C$ such that 
\begin{itemize} 
\item $|P_\ell^n(x)/P_\rho^n(x)| < C$ 
for all $|z| \geq 1$, all $\ell$, and all $n \geq N_0$, 
\item $\sum_{j=0}^k j!a_{\rho,j}\binom n j \neq 0$,
\item the zeros of $f_n^T$ are in $D(0,1/2)$. 
\end{itemize}
Hence the estimate in \eqref{tratt} can be made uniform in $n$. Indeed, we can choose $R_0=C+1$.  
\end{proof}

\begin{remark}\label{inrem}
Note that by item (2) of Theorem~\ref{th:generalGeN}, if $T$ is a linear operator and $\Omega \subseteq \bC$ is closed and unbounded,
then $\Omega$ is $T_n$-invariant if and only if it is $T_\ell$-invariant for all $\ell \leq n$. 
Indeed if $f(z)$ has degree $\ell \leq n$ we may take a sequence $\{w_j\}_{j=1}^\infty$ in $\Omega$
for which $|w_j| \to \infty$ as $j \to \infty$. Then the zeros of 
\[
T(f) = \lim_{j \to \infty} T[ (1-{x}/{w_j})^{n-\ell}f(z)]
\]
is in $\Omega$ by Hurwitz' theorem. 
\end{remark}

The following important notion can be found in \cite[Def. 1]{BorceaBranden2009}. 
\begin{definition}\label{def:stablePoly}
A polynomial $f(z_1,\dotsc,z_\ell) \in \bC[z_1,\dotsc,z_\ell]$ is called \defin{stable} if 
for all $\ell$-tuples $(z_1,\dotsc,z_\ell) \in \bC^\ell$ with $\Im(z_j) > 0$, $1 \le j \le \ell$, 
one has $f(z_1,\dotsc,z_\ell)\neq 0$. 
\end{definition}

\begin{proposition}\label{prop:halfPlaneCharaterization}
Take a closed half-plane given by $H= \{ a x +b : \Im\, x \leq 0\}$, $(a,b)\in \bC^2,\, a\neq 0$ 
and let $T = \sum_{j=0}^k Q_j(x) \diff{j}$ be a differential operator. 
Then the  following facts are equivalent: 
\begin{enumerate}
\item The set of positive integers $n$ for which $H$ is $T_n$-invariant is unbounded,
\item $H$ is $T_n$-invariant for all $n \geq 0$, 
\item The polynomial $\sum_{j=0}^k Q_j(ax+b)(-y/a)^j$ considered as an element in $\bC[x,y]$ is a stable polynomial in $(x,y)$.
\end{enumerate}
\end{proposition}
\begin{proof}
By Remark~\ref{inrem} we see that (1) and (2) are equivalent.
Now, (2) is equivalent to the fact  that the operator $S : \bC[x] \rightarrow \bC[x]$ defined by 
\[
S(f)(x) = T(f(\phi^{-1}(x)))(\phi(x)),
\]
where $\phi(x)= ax+b$, preserves stability. 
The operator $S$ is again a differential operator, so the equivalence of (2) and (3) now follows from \cite[Theorem 1.2]{BorceaBranden2010}. 
\end{proof}

\begin{example}
Consider the operator $T:\bC_n[x] \to \bC[x]$ given by
\begin{equation}
T = (x^2-x^3)\diff{3} + ( x + x^2 ) \diff{2} + 2x  \diffx -6.
\end{equation}
When $n=3$,  we have that for every $z \in \bC$, 
\begin{equation}
 T\left[ (x-z)^3 \right] = 12\left(x - z^2 \right) \left( x - z/2 \right).
\end{equation}
In particular, if $z$ lies in a $T_3$-invariant set, then $z^2$ is also in the set.
Thus, there are no large $T_3$-invariant disks.
However, this does not violate Theorem~\ref{genethm}: since
the $3$rd Fuchs index of $T$ is 0, but $P_\rho^n(x) = -6(1+2x)$.
Hence, the operator  $T$ is degenerate for $n=3$ and Theorem~\ref{genethm}  does not apply.
\end{example}

\subsection{Description of the limiting minimal set \texorpdfstring{$\minvset{\infty}$}{}.}

Recall that in Corollary~\ref{cor:natural}, we proved that whenever
the leading coefficient $Q_k(x)$ of an operator $T$ is has positive degree, 
then there is a minimal invariant set $\minvset{\infty}$ containing 
the convex hull of the roots of $Q_k(x)$.
Furthermore, if $T$ is non-degenerate, Corollary~\ref{cor:nonDegenerateStableDisks}
implies that $\minvset{\infty}$ is compact. The next result of the third author is the main motivation for Theorem~\ref{thm:limit}.

\begin{THEO}[{See \cite[Thm. 9]{Shapiro2010}}] \label{lm:collapse} 
Given a non-degenerate operator $T$ as in \eqref{eq:main} and $\eps>0$, 
there exists a positive integer $n_\eps$ such that for any $n>n_\eps$ 
and any polynomial $p$ of 
degree $n$ with all roots in $\Conv(Q_k)$, 
all roots of $T(p)$ lie in the $\eps$-neighborhood  of $\Conv(Q_k)$.  
\end{THEO}

The main technical tool in the proof of Thereom~\ref{thm:limit}  is Theorem~\ref{th:DISCS} which is of independent interest. It extends the previous Theorem~\ref{genethm}. For the proof we will make use of the following alternative ``symbol theorem'' which follows from \cite[Thm.7]{BorceaBranden2009}.

\begin{proposition}\label{prop:bb09alt}
Let $T : \bC_n[x] \rightarrow \bC[x]$ be a linear operator of rank greater than one, and let $D$ be a closed disk in $\mathbb{C}$. Then $D$ is $T_n$-invariant if and only if $G_n(x,y) \neq 0$ whenever $x \in D^c$ and $y \in D$, where 
\[
G_n(x,y)= T\big((x-y)^n\big)= \sum_{i=0}^n \binom n i T(x^i)y^{n-i}.
\]
\end{proposition}

\begin{theorem}\label{th:DISCS} Given a non-degenerate operator $$T=Q_k(x) \frac{d^k}{dx^k} +Q_{k-1}(x)\frac{d^{k-1}}{dx^{k-1}}+\cdots+Q_0(x),$$ let $D$ be any closed disk that contains $\Conv(Q_k)$, and is such that the distance between $\Conv(Q_k)$ and the boundary of $D$ is positive. Then $D$ is $T_n$-invariant for all sufficiently large degrees $n$. 
\end{theorem}

\begin{proof}

By Proposition~\ref{prop:bb09alt}, $D$ is $T_n$-invariant if the polynomial
\[
G_n(x,y)= \sum_{j=0}^k (n)_j\cdot Q_j(x)\cdot (x-y)^{n-j}
\]
is nonzero whenever $x \in D^c$ and $y \in D$. 

For fixed $j<k$ and $y \in D$, the polynomial (in $x$),
\[
\frac {Q_j(x) \cdot (x-y)^{k-j}}{Q_k(x)}
\]
is uniformly bounded on $D^c$. This is because the degree of the numerator is less than or equal to the degree of the denominator, and the zeros of $Q_k(x)$ have positive distance to $D^c$. By compactness of $D$ there is a constant $C$ such that 
\[
\left| \frac {Q_j(x) \cdot (x-y)^{k-j}}{Q_k(x)} \right | \leq C, \ \ \ \text{ for all } x \in D^c, y \in D.
\]
Hence, there is a constant $L$, independent of $n$, such that 
\[
\left |\frac {G_n(x,y)}{Q_k(x)\cdot (x-y)^{n-k} \cdot (n)_k} -1\right| < \frac L n,   \ \ \ \text{ for all } x \in D^c, y \in D.
\]
It follows that for $n$ sufficiently large, $G_n(x,y)$ is nonzero whenever $x \in D^c$ and $y \in D$. 
\end{proof} 

\begin{theorem}\label{thm:limit}
If $T$ is non-degenerate, then $M_\infty^T=Conv(Q_k)$. 
\end{theorem}
\begin{proof}
We assume $Conv(Q_k)$ is not a line or a point. The proofs for those cases are similar. 

Let $\epsilon>0$. For each side $L$ of the polygon $Conv(Q_k)$, let $D_\epsilon(L)$ be a disc containing $Conv(Q_k)$ such that the distance between $L$ and the boundary of $D_\epsilon(L)$ is at most $\epsilon$ and at least $\epsilon/2$, see Fig.~\ref{figNN}. By Theorem~\ref{th:DISCS},  $D_\epsilon(L)$ is $T_n$-invariant for all $n \geq N(L,\epsilon)$, where $N(L,\epsilon)$ is a positive integer. But then 
\[
K_\epsilon= \bigcap_{L} D_\epsilon(L)
\]
is $T_n$-invariant for all $n \geq N(\epsilon)$, where $N(\epsilon)= \max_L N(L,\epsilon)$. Clearly $K_\epsilon \to Conv(Q_k)$.
\begin{figure}
\centering
\includegraphics[scale=0.6]{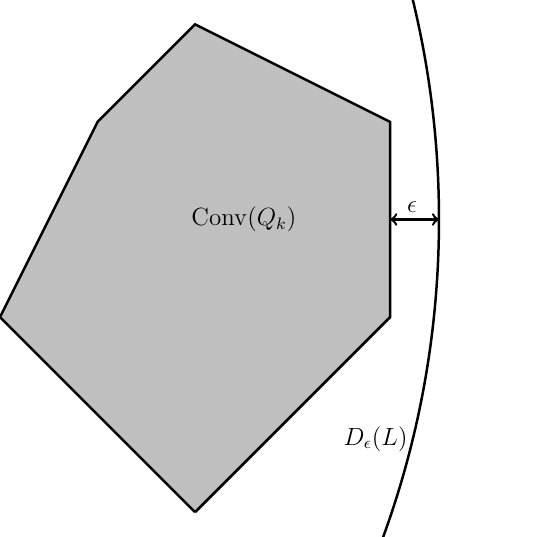}
\caption{Illustration to the proof of Theorem~\ref{thm:limit}} \label{figNN}
\end{figure}
\end{proof}

Let us now describe a special class of non-degenerate operators for which all $\minvset{\ge n},\, n=0,1,\dots , $ coincide with each other and with the fundamental polygon $\Conv(Q_k)$. 

\begin{proposition}\label{prop:Lame}
 Take a non-degenerate operator of the form $T=Q_k(x)\frac{d^k}{dx^k}+Q_{k-1}(x)\frac{d^{k-1}}{dx^{k-1}}$ satisfying the condition 
\begin{equation}\label{eq:Polya}
\frac{Q_{k-1}(x)}{Q_k(x)}=\sum_{i=1}^{\deg Q_k}\frac{\kappa_i}{x-x_i},
\end{equation}
where $\kappa_i\ge 0$ and $\{x_1,\dots, x_{\deg Q_k}\}$ is the set of all roots of $Q_k(x)$.  Then, 
\[
\minvset{}=\minvset{\ge 1}=\minvset{\ge 2}=\dots =\minvset{\infty} = \Conv (Q_k).
\]
\end{proposition}
\begin{proof}
By item (3) of Theorem~\ref{th:generalGeN}, it suffices to show that under our assumptions on $T$, $\Conv(Q_k)$
is a $T$-invariant set. Moreover by Gauss--Lucas theorem, for $T=Q_k(x)\frac{d^k}{dx^k}+Q_{k-1}(x)\frac{d^{k-1}}{dx^{k-1}}$ satisfying \eqref{eq:Polya},  it suffices to show that $\Conv(Q_k)$ is $\widetilde T$-invariant where 
$\widetilde T=Q_k(x)\diffx + Q_{k-1}(x)$. 
Assume now that $p(x)$  is an arbitrary polynomial of some degree $n$ whose 
roots  $r_1, \dots , r_n$ lie in $\Conv(Q_k)$ and consider $q=\widetilde T(p)$. 
We want to show that  $q(z)\neq 0$ for any $x\in \bC \setminus \Conv(Q_k)$. 
Assume $q(x)=0$ which is equivalent to 
\begin{align}\label{eq:bla}
Q_k(x)p^\prime(x)+Q_{k-1}(x)p(x)=0 \quad \Leftrightarrow 
\quad \frac{p^\prime(x)}{p(x)} = -\frac{Q_{k-1}(x)}{Q_k(x)}.
\end{align}
The latter expression is equivalent to 
\[
\sum_{j=1}^n\frac{1}{x-r_j} = -\sum_{i=1}^{\deg Q_k}\frac{\kappa_i}{x-x_i},
\]
where $\{x_1,\dots, x_{\deg Q_1}\}$ is the set of roots of $Q_k$ and $\kappa_i\ge 0$. 
Assuming that $x \notin \Conv(Q_k)$, choose a line $L$ separating $z$ from $\Conv(Q_k)$. 
By our assumptions, $L$ separates $x$ from all $r_j$'s and all $x_i$'s.
Because of this and taking into account the signs, one can easily conclude that the 
left-hand side of the latter expression is a complex number pointing from $x$ to the 
half-plane not containing $x$ and the right-hand side does the opposite. 
Therefore, \eqref{eq:bla} can not hold if $x\notin \Conv(Q_k)$. 
\end{proof}

A special case of Proposition~\ref{prop:Lame} when $Q_k(x)$ is a real-rooted polynomial
follows from more general results of \cite{Branden2011}.

\section{Exactly solvable and degenerate operators: basic facts}\label{sec:degenerate}

\subsection{Preliminaries on exactly solvable operators} 

In this section we will need the following information, see e.g. \cite{Bergkvist2007}.   

\medskip
Given an exactly solvable operator $T$, observe that  for each non-negative integer $j$, 
 \begin{equation}\label{eq:spectrum}
 T(x^j)=\la_j^T x^j+\text{ lower order terms}.
 \end{equation}
   Define the  \defin{spectrum}  of an exactly solvable $T$  as the sequence  $\La^T \coloneqq \{\la_j^T\}_{j=0}^\infty$ of complex 
 numbers.

\begin{LEMMA}[See \cite{MassonShapiro2001}]\label{lm:eigenpolys}

For any exactly solvable operator $T$ and any sufficiently large positive integer $n$,
there exists a unique (up to a constant factor) eigenpolynomial $p_n^T(x)$ of $T$ of degree $n$.
Additionally, the eigenvalue of $p_n^T$ equals $\la_n^T$, where $\la_n^T$ is given by \eqref{eq:spectrum}.
\end{LEMMA}

One can easily show that for any exactly solvable operator $T$, the sequence  $\{|\la_m^T|\}$ is  monotone increasing to $+\infty$ which implies that for any 
sufficiently large positive integer $m$, $\vert \la_j^T\vert < \vert \la_m^T\vert$ for $0\le j<m$. 

\begin{remark} 
In addition to Lemma~\ref{lm:eigenpolys}, observe that for any exactly solvable operator $T$ as in \eqref{eq:main}
and any non-negative integer $n$, $T$ has a basis of eigenpolynomials in the linear space $\bC_n[x]$ consisting 
of all univariate polynomials of degree at most $n$.  This follows  immediately from e.g., 
the fact that $T$ is triangular in the monomial basis $\{1,x,\dots, x^n\}$. 
In other words, even if $T$ has a multiple eigenvalue it has no Jordan blocks. However,  the eigenpolynomial in the respective degree is no longer unique.  
A simple example of such situation occurs for $T=x^k\frac{d^k}{dx^k}$ in which case any polynomial of degree less than $k$ lies in the kernel. 
\end{remark}  

In what follows, we will use the following result.

\begin{proposition}\label{pr:rootsEigenpolys}
Given an exactly solvable operator $T$ as in \eqref{eq:main} and any invariant 
set $S\in \invset{\ge n}$, one has that $S$ must contain the union of all roots of the 
eigenpolynomial $p_m^T$ satisfying two conditions:  $n\le m$ and $\vert \la_j^T \vert< \vert \la_m^T\vert$ where $0\le j<m$. 
The latter fact implies that $S$ contains the union of all roots of all eigenpolynomials of sufficiently large degrees. 
\end{proposition} 

\begin{proof}
Indeed, as we mentioned above that  the sequence $\{\vert \la_n^T\vert \}$ will be strictly 
increasing to $+\infty$ starting from some positive integer $\kappa^T$. Choose some $m\ge n$ such 
that $m\ge \kappa^T$ which implies that $\vert \la_m^T \vert> \vert \la_j^T\vert$ for $0\le j<m$ and that $\{p_0^T, p_1^T, \dots, p_m^T\}$ isa basis in 
the space $\bC_m[x]$ of all polynomials of degree at most $m$. 

Pick a polynomial $q$ of degree $m$ whose roots belong 
to $S$ and expand it as $q(x)=\sum_{j=0}^m a_j p_j^T(x)$ with $a_m\neq 0$. 
Repeated application of $T$ to $q$ gives
\begin{equation}\label{eq:iter}
T^{\circ \ell}(q)=\sum_{j=0}^m a_j \la_j^\ell p_j^T(x)=\la_m^\ell\sum_{j=0}^ma_j\left(\frac{\la_j}{\la_m}\right)^\ell p_j^T(x).
\end{equation} 
Since $S \in \invset{\ge n}$, all roots of $T^{\circ \ell}(q)$ belong to $S$. 
By our assumption and disregarding the common factor $\la_m^\ell$,  the polynomial 
in the right-hand side of \eqref{eq:iter} equals $a_m p_m^T(x)$ plus some 
polynomial of degree smaller than $m$ whose coefficients tend to $0$ as $\ell\to \infty$.  
 Since $a_m\neq 0$ the roots of the polynomials 
in the right-hand side of \eqref{eq:iter} tend to those of $p_m^T$ implying 
that the latter roots must necessarily belong to $S$. 
\end{proof}

\subsection{Preliminaries on degenerate operators} 

An important although not very complicated result about degenerate operator 
which partially follows from our previous considerations is as follows.  

\begin{proposition}\label{pr:degen}
 If $T$ is a degenerate operator, 
then for any non-negative $n$, every set in $\invset{\geq n}$ is unbounded and, 
therefore is $T$-invariant.  
\end{proposition} 

\begin{proof} 
We only need to show the unboundedness  since $T$-invariance follows from the unboundedness by item~\upshape{(2)}
of Theorem~\ref{th:generalGeN}.
Let us start with the special case of degenerate exactly solvable operators.  
(These operators and their invariant sets are the main object of study of our sequel paper \cite{AlBrSh2}.) 

Any exactly solvable operator $T$ preserves the degree of a generic polynomial it acts upon 
and has a unique (up to a constant factor) eigenpolynomial $p_n^T(x)$ of any sufficiently 
large degree $n$, see Lemma~\ref{lm:eigenpolys} and  \cite[Lemma 1]{Bergkvist2007}.
Moreover,  if $r_n$ denotes the maximum of 
the absolute value  of the roots of $p_n(x)$, then for any degenerate exactly 
solvable $T$, $\lim_{n\to \infty} r_n=+\infty$, see \cite[Theorem 1]{Bergkvist2007}.

By Proposition~\ref{pr:rootsEigenpolys}  for any exactly solvable operator $T$, any set $S\in \invset{\geq n}$ must 
contain the union of all roots of all eigenpolynomials $p_m^T(x)$ for all sufficiently large $m$, we conclude that 
any such $S$ is necessarily unbounded. 


Assume now that $T$ has a positive 
Fuchs index $\rho \coloneqq \rho_T>0$. Consider the operator $T^\prime=\frac{d^{\rho}}{dx^{\rho}}\circ T$.
If $T$ is degenerate, then $T^\prime$ is a degenerate exactly solvable operator. 
By the Gauss--Lucas theorem, every $S\in \invset{\geq n}$ belongs to $\mathcal {I}_{\geq n}^{T^\prime}$. 
Since every subset $S^\prime \in \mathcal{I}_{\ge n}^{T^\prime}$ is unbounded by the above argument, 
we have settled the case $\rho>0$. 

Assume finally, that $T$ is a degenerate operator with $\rho<0$. 
Consider a family of operators 
\[
T^\prime_a = (x-a)^{-\rho}\cdot T, 
\]
where 
$a\in \bC$.  
Since under our assumptions, $-\rho$ is a positive integer,  $T_a$ is a degenerate exactly solvable operator for any $a$. 
Given  $S\in \invset{\geq n}$, choose $a\in S$. Then $S\in  \mathcal{I}_{\ge n}^{T^\prime_a}$ 
and is therefore unbounded by the previous reasoning. 
\end{proof}

\section {(Tropical) algebraic preliminaries and three types of Newton polygons}\label{sec:trop} 

In our study of invariant sets for degenerate operators we will need some classical results about root 
asymptotics of bivariate polynomials in the spirit  of modern tropical geometry, see \cite[Section 38, ]{Chebotarev1948} and   \cite[Ch.~4]{Walker1950}.  These results will be used in \S~\ref{sec:deg}. 

\medskip
We start by introducing the domination partial order on points in $\bR^2$, 
Namely, we say that a point $p=(u,v)\in \bR^2$ \defin{dominates} a point $p^\prime=(u^\prime, v^\prime)$ 
if $u\ge u^\prime$ and $v\ge v^\prime$. 
Given a subset $S\subseteq \bR^2$, we call by its \defin{northeastern border} $\NE_S$ 
the set of all points in $S$ which are not dominated by other points in $S$. 
Observe that $\NE_S$ can be empty if $S$ is non-compact, but for compact $S$, $\NE_S$ is always nonempty. 
Furthermore, if $S$ is both compact and convex then $\NE_S$ is contractible.

Given a bivariate polynomial $R(u,v)=\sum_{(i,j)\in \Theta} a_{i,j}u^iv^j$, 
denote by $Conv(R)\subset \bR^2$ its \textcolor{blue}{\emph{Newton polygon}}, i.e. the convex hull of the set of exponents $(i,j)\in \Theta$. 
The northeastern border  of $Conv(R)$ will be  denoted by $\NE_R$, see examples in Figure~\ref{fig:neBoundary} and
Figure~\ref{fig:newtonCases}.
By the above, $\NE_R$ is connected and contractible. 
The point of $\NE_R$ with the maximal value of $u$ will be called the \defin{eastern vertex} and denoted by $V_{e}$ and 
the point of $\NE_R$ with the maximal value of $v$ will be called the \defin{northern vertex} and denoted by $V_{n}$. 
 The set $\NE_R$ coincides with a point if and only if $V_{e}=V_{n}$. 
 Notice that every edge of the boundary of $Conv(R)$ included in $\NE_R$ has a negative slope. 
 Finally, denote by $R^{ne}(u,v)$ the restriction of $R(u,v)$ to the subset $\Theta^{ne}\subseteq \Theta$ consisting of all 
 monomials whose exponents are the vertices of $\NE_R$.  
 We will call $R^{ne}(u,v)$ the \defin{northeastern part} of $R(u,v)$. 
 
 \begin{remark} 
 Observe that for any bivariate $R(u,v)$ and $\al, \beta \in \bC$, the change of variables of the form $u=\tilde u +\al, v=\tilde v +\beta$ 
 does not change neither $\NE_R$  nor $R^{ne}(u,v)$.  
 \end{remark}
 
 Given an arbitrary bivariate polynomial 
 \[
 R(u,v)=\sum_{(i,j)\in \Theta} a_{i,j}u^iv^j=\sum_{j=0}^mR_j(v)u^j
 \]
 and some number $w\in \bC$, denote by $\mathcal U_R(w)$ the set of zeros of the equation $R(u,w)=0$ in the variable $u$ 
 considered as the divisor in $\bC$, i.e. zeros are counted with multiplicities. 
 Here  $m$ is the degree of $R$ w.r.t.\ $u$. 
 Assume that the parameter $w$ runs over the portion of the positive half-axis $[\kappa,+\infty)$ which contains no 
 root of $R_m(v)$; one can always choose $\kappa$ sufficiently large 
 so that the latter condition is satisfied. (Obviously, for all $w\in [\kappa,+\infty)$, the degree of the divisor $\mathcal U_R(w)$ equals $m$.) 
 We define the subdivisor $\mathcal U^\infty_R(w)\subset \mathcal U_R(w)$ as the  set of all roots $u(w)$
 whose absolute values tend to $\infty$ when $w$ tends to $+\infty$ along the positive half-axis. 
 Notice that $\mathcal U^\infty_R(w)$ is well-defined for all sufficiently large positive $\widetilde \kappa> \kappa$, since 
 there exists $\widetilde \kappa$ such that for any $w\in [\widetilde \kappa, +\infty),$ the absolute value 
 of every root in $\mathcal U^\infty_R(w)$ will be strictly larger than the absolute value of any root in 
 the complement $\mathcal U_R(w) \setminus \mathcal U^\infty_R(w)$. 
 
 Our next goal is to describe  $\mathcal U^\infty_R(w)$ in terms of $R^{ne}(u,v)$. 
 In what follows we will frequently use the following statement.

Given an arbitrary bivariate polynomial $R(u,v)$ whose $\NE_R$ is not a single point, decompose $\NE_R$ into the (disjoint) 
union of consecutive edges $\NE_R=\cup_{s=1}^h e_s$ covering $\NE_R$ from north to east. 
That is $e_1$ starts at $V_{n}$, $e_h$ ends at $V_{e}$, and each $e_s$ is adjacent to $e_{s+1}$, see Figure~\ref{fig:neBoundary}.
The absolute values of the slopes of $e_1,\dotsc, e_h$ are strictly increasing. 
The following statement can be easily deduced from the known results of \cite[Section 38, Th.~63--66]{Chebotarev1948}, 
and \cite[Ch.~4, Sections 3 and 4]{Walker1950}. 
(To use the latter results, one has to substitute $u$ and $v$ by $u^{-1}$ and $v^{-1}$ respectively.) 

\begin{proposition}\label{pr:Newton}  
 The degree of the  divisor $\mathcal U^\infty_R(w)$ is equal to $i_{e}-i_{n}$ where $V_{e}=(i_{e},j_{e})$ 
 and $V_{n}=(i_{n},j_{n})$. In other words, $\deg\;  \mathcal U^\infty_R(w)$ equals 
 the length of the projection of $\NE_R$ onto the $u$-axis.  
 
 Additionally, $\mathcal U^\infty_R(w)$ splits into $h$ subdivisors $\mathcal U^\infty_1(w),\dots, \mathcal U_h^\infty(w)$ 
 corresponding to the edges $e_1,\dots, e_h$ respectively; the degree of $\mathcal U^\infty_s(w),\;s=1,\dots, h$ equals  
 the length of the projection of $e_s$ on the $u$-axis. All zeros in the divisor $\mathcal U^\infty_s(w)$ 
 have the asymptotic growth $u\sim \eps w^{sl_s}$ where $sl_s$ is the absolute value of the slope of $e_s$.

Possible values of $\eps$ can be found by substituting $ \eps w^{sl_s}$ in the 
restriction of $R(u,v)$ to the monomials contained in the edge $e_s$ and finding the non-vanishing roots of this restriction.  
\end{proposition}

\begin{definition} \label{def:imp} 
Given an arbitrary bivariate polynomial $R(u,v)$ whose northeastern border  $\NE_R$ is not a single point, 
we will call the slopes of edges in $\NE_R$ the \defin{characteristic exponents} of $R(u,v)$.  
For a given edge $e_s\in \NE_R$, all possible values of $\eps$ corresponding to the 
restriction of $R(u,v)$ to this edge will be called the \defin{leading constants} corresponding 
to (the characteristic exponent of) $e_s$. 
The union of all leading constants of $R(u,v)$ will be denoted by $\Upsilon_R$.
\end{definition}

\begin{example}\label{ex:R} 
To illustrate Proposition~\ref{pr:Newton} and Definition~\ref{def:imp},
take 
\[
R(u,v)= u^8+u^7v^2 +u^5v^4+(5+7\sqrt{-1})u^3v^6-23uv^7.
\]
One can easily check that all monomials in $R(u,v)$ belong to $\NE_R$ which consists of three 
edges $e_1,e_2, e_3$ connecting $(1,7)$ with $(3,6)$,  $(3,6)$ with $(7,2)$, and $(7,2)$ with $(8,0)$ resp.  
(The exponent  $(5,4)$ of the second monomial belongs to $e_2$.) 
Degree of  $\mathcal U^\infty_R(w)$  equals $8-1=7$. Restriction $R_1(u,v)$ of $R(u,v)$ to $e_1$ 
is given by $(5+7\sqrt{-1})u^3v^6-23uv^7$. 
Its nontrivial zeros with respect to the variable $u$ are given by $(5+7\sqrt{-1})u^2-23w=0$. 
Thus for two roots from $\mathcal U_1^\infty(w)$, $u\sim \eps w^{1/2}$ where $\eps$ are the two roots 
of the equation $(5+7\sqrt{-1})\eps^2-23=0$. They are approximately equal to $-1.45392\pm 0.748212$. 
(The absolute value of the slope of $e_1$ equals $\frac{1}{2}$.) 
 Restriction $R_2(u,v)$ of $R(u,v)$ to $e_2$ is given by $u^7v^2+u^5v^4+(5+7\sqrt{-1})u^3v^6$. 
 Its nontrivial zeros with respect to the variable $u$ are given by $u^4+u^2w^2+(5+7\sqrt{-1})w^4=0$; 
 we have substituted $w$ instead of $v$ here to keep our notation. 
 Thus for $4$ different roots belonging to $\mathcal U_2^\infty(w)$,  we have $u\sim \eps w$ where $\eps$ are the four 
 roots  of the equation $\eps^4+\eps^2+5+7\sqrt{-1}=0$. 
 These are approximately equal to  $-1.22651\pm 0.961446\sqrt{-1}$ and $-0.809831\pm 1.58673 \sqrt{-1}$. 
 (The absolute value of the slope of $e_2$ equals $1$.) 
 Finally,  the restriction $R_3(u,v)$ of $R(u,v)$ to $e_3$ is given by $u^8+u^7v^2$ which gives $u=-v^2$. 
 (The absolute value of the slope of $e_3$ equals $2$.) 
 Summarizing, we get that  $\Upsilon_R$ consists of $6$ complex numbers 
 approximately given by $\{-1, -1.22651\pm 0.961446\sqrt{-1}, -0.809831\pm 1.58673 \sqrt{-1}, -1.45392\pm 0.748212\}$. 
 Its convex hull contains $0$ as its interior point.
\end{example}

\begin{figure}[!ht]
\begin{tikzpicture}[baseline=(current bounding box.center)]
\draw[step=1em, lightgray, very thin] (0,0) grid (9em, 9em);
\draw[x=1em,y=1em,black,->] (0,0)--(9,0);
\draw[x=1em,y=1em,black,->] (0,0)--(0,9);
\begin{scope}
\draw[black, very thick,x=1em,y=1em] (1,7)--(3,6)--(5,4)--(7,2)--(8,0);
\fill[black,x=1em,y=1em] (1,7) circle (0.15em);
\fill[black,x=1em,y=1em] (3,6) circle (0.15em);
\fill[black,x=1em,y=1em] (5,4) circle (0.1em);
\fill[black,x=1em,y=1em] (7,2) circle (0.15em);
\fill[black,x=1em,y=1em] (8,0) circle (0.15em);
\node[x=1em,y=1em] at (0,9.5) {v};
\node[x=1em,y=1em] at (9.5,0) {u};
\node[x=1em,y=1em] at (1.2,7.8) {$V_n$};
\node[x=1em,y=1em] at (8.0,-1.0) {$V_e$};
\node[x=1em,y=1em] at (2.5,7.0) {$e_1$};
\node[x=1em,y=1em] at (5.7,4.5) {$e_2$};
\node[x=1em,y=1em] at (8.2,1.4) {$e_3$};
\end{scope}
\end{tikzpicture}
 \caption{The northeastern border of the Newton polygon 
of 
 $R(u,v)= u^8+u^7v^2 +u^5v^4+(5+7\sqrt{-1})u^3v^6-23uv^7$, see Example~\ref{ex:R}. (The Newton polygon itself is obtained by adding an edge connecting $V_n$ with $V_e$.)
 }\label{fig:neBoundary}
\end{figure}
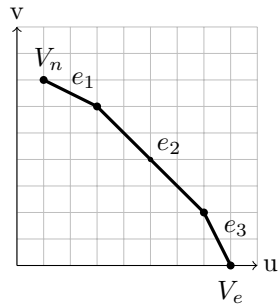

\begin{corollary}\label{cor:internal} In the above notation, for a given bivariate polynomial $R(u,v)$, the family of  convex hulls of $ \mathcal U^\infty_R(w)$ converges to  $\bC$ when $w\to +\infty$ if and only if the convex hull of $\Upsilon_R$ contains $0$ as its interior point.
\end{corollary}

\begin{proof} (Sketch) 
This statement is rather obvious since if $0$ is an interior point of the convex hull  of $\Upsilon_R$, then the roots in $\mathcal U^\infty_R(w)$ will be asymptotically moving to infinity when $w\to +\infty$ in the directions prescribed by all values of $\epsilon\in \Upsilon_R$  and their convex hull will contain the  disk of any given radius  centered at $0$ for sufficiently large $w$.  
\end{proof}

Let us fix a connected contractible piecewise linear curve $\NE\subset \bR^2$ with 
integer vertices consisting of pairwise non-dominating points, see Figure~\ref{fig:neBoundary}.
In other words, $\NE$ is a piecewise linear path with integer vertices whose edges have 
negative slopes whose absolute values increase when moving down along the path. 
Denote by $Pol(\NE)$ the set of  all bivariate polynomials whose northeastern border coincides with a given $\NE$. 
(In particular, we assume that all coefficients at the corners/endpoints of $\NE$ are non-vanishing. 
$Pol(\NE)$ is a Zariski-open subset of a finite-dimensional linear space of bivariate polynomials.) 
Recall that the \defin{integer length} of a closed straight interval $I \subset\bR^2\supset \bZ^2$ 
is the number of points from $\bZ^2$ contained in $I$, i.e. the number of integer points belonging to $I$.

\begin{definition}\label{def:newtonCases} 
Given $\NE\subset \bR^2$ as above, we call it 

\begin{enumerate}[label={\upshape(\roman*)}]
\item 
 \defin{defining} if there exists an edge in $\NE$ with the 
slope $-\al/\beta$ where $\alpha$ and $\beta$ are coprime positive integers and $\beta \ge 3$; 

\item 
\defin{almost defining} 
if there are no edges as in \textrm{(i)}, but there either

\begin{enumerate}
\item 
 exists at least 
one edge in $\NE$ with the  slope $-\al/2$ and whose integer length  is larger than $2$,  or 

\item 
there exist at least two edges with the  slope $-\al/2$ and integer length at least $2$;
\end{enumerate}

\item 
\defin{non-defining}
in the remaining case i.e., when either all edges of $\NE$ have negative integer slopes or all edges but one 
have negative integer slopes  and the remaining edge  has a 
negative half integer slope and integer length $2$. 
\end{enumerate}

\end{definition}

\begin{figure}[!ht]
\begin{tikzpicture}[baseline=(current bounding box.center)]
\draw[step=1em, lightgray,opacity=0.2, very thin] (-1.2,0) grid (24em, 15em);
\begin{scope}[xshift=0em,yshift=8em]
\draw[black, very thick,x=1em,y=1em] (0,7) -- (3,6)--(4,2);
\draw[black, dashed,x=1em,y=1em] (4,2)--(0,0)--(-2,1)--(-3,5)--(0,7);
\node[x=1em,y=1em] at (0.5,3.5) {\tiny{(i) Defining}};
\node[x=1em,y=1em] at (1.8,7.0) {\tiny{$-\tfrac{1}{3}$}};
\node[x=1em,y=1em] at (4.5,4.5) {\tiny{$-\tfrac{4}{1}$}};
\end{scope}
\begin{scope}[xshift=9em,yshift=8em]
 \draw[black, very thick,x=1em,y=1em] (0,7) -- (4,5)--(5,2);
\draw[black, dashed,x=1em,y=1em] (5,2)--(1,0)--(-1,1)--(-2,4)--(0,7);
\node[x=1em,y=1em] at (1,3.5) {\tiny{(ii) Almost def.}};
\node[x=1em,y=1em] at (2.8,6.5) {\tiny{$-\tfrac{2}{4}$}};
\node[x=1em,y=1em] at (5.2,4.0) {\tiny{$-\tfrac{3}{1}$}};
\end{scope}
\begin{scope}[xshift=19em,yshift=8em]
 \draw[black, very thick,x=1em,y=1em] (0,7) -- (2,6)--(4,3);
\draw[black, dashed,x=1em,y=1em] (4,3)--(4,2)--(0,0)--(-2,1)--(-3,5)--(0,7);
\node[x=1em,y=1em] at (0.5,3.5) {\tiny{(ii) Almost def.}};
\node[x=1em,y=1em] at (2.2,6.5) {\tiny{$-\tfrac{1}{2}$}};
\node[x=1em,y=1em] at (3.6,5.0) {\tiny{$-\tfrac{3}{2}$}};
\end{scope}
\begin{scope}[xshift=1em,yshift=0em]
 \draw[black, very thick,x=1em,y=1em] (0,7) -- (2,5)--(3,2);
\draw[black, dashed,x=1em,y=1em] (3,2)--(0,0)--(-2,1)--(-3,5)--(-1,7)--(0,7);
\node[x=1em,y=1em] at (0,3.5) {\tiny{(iii) Non-def.}};
\node[x=1em,y=1em] at (1.5,6.5) {\tiny{$-\tfrac{2}{2}$}};
\node[x=1em,y=1em] at (3.2,3.5) {\tiny{$-\tfrac{3}{1}$}};
\end{scope}
\begin{scope}[xshift=10em,yshift=0em]
 \draw[black, very thick,x=1em,y=1em] (0,7) -- (2,6)--(3,2);
\draw[black, dashed,x=1em,y=1em] (3,2)--(0,0)--(-2,1)--(-3,5)--(0,7);
\node[x=1em,y=1em] at (0,3.5) {\tiny{(iii) Non-def.}};
\node[x=1em,y=1em] at (1.8,6.8) {\tiny{$-\tfrac{1}{2}$}};
\node[x=1em,y=1em] at (3.3,3.8) {\tiny{$-\tfrac{4}{1}$}};
\end{scope}
\begin{scope}[xshift=18em,yshift=0em]
 \draw[black, very thick,x=1em,y=1em] (0,7) -- (2,5)--(4,2);
\draw[black, dashed,x=1em,y=1em] (4,2)--(3,0)--(0,0)--(-2,1)--(-3,5)--(-1,7)--(0,7);
\node[x=1em,y=1em] at (0,3.5) {\tiny{(iii) Non-def.}};
\node[x=1em,y=1em] at (1.8,6.2) {\tiny{$-\tfrac{2}{2}$}};
\node[x=1em,y=1em] at (3.8,3.8) {\tiny{$-\tfrac{3}{2}$}};
\end{scope}
\end{tikzpicture}
 \caption{Examples of defining/almost defining/non-defining Newton polygons, see Definition~\ref{def:newtonCases}.
 The slopes of the edges of the northeasten boundary are shown as fractions,
 such that the length of the projection is  the respective denominator.
 }\label{fig:newtonCases}
\end{figure}
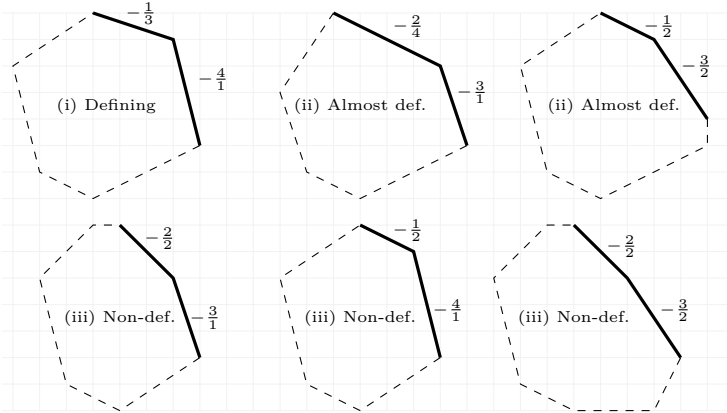

\begin{definition}\label{def:Newton} 
A Newton polygon $N\subset \bR^2$ is called \defin{defining/almost defining/non-defining} if its northeastern border 
contains at least one edge and is  \defin{defining/almost defining/non-defining} respectively. 
\end{definition}

In Figure~\ref{fig:newtonCases}, we show examples of Newton polytopes illustrating Definition~\ref{def:newtonCases} and  Definition~\ref{def:Newton}.

\begin{proposition} \label{pr:imp} 
Given $\NE\subset \bR^2$ as above, the  convex hull of $\mathcal U^\infty_R(w)$ converges to $\bC$,  when $w\to +\infty$

\begin{enumerate}[label={\upshape(\roman*)}]
\item 
 for any $R\in Pol(\NE)$ if $\NE$ is defining;

\item 
for generic $R\in Pol(\NE)$ if $\NE$ is almost defining;

\item 
if $\NE$ is non-defining there is a full-dimensional subset of $Pol(\NE)$ 
for which the  convex hull of $\mathcal U^\infty_R(w)$ converges to $\bC$ when $w\to +\infty$ and 
the complement of the latter set in $Pol(\NE)$ is also full-dimensional. 
\end{enumerate}
\end{proposition} 

\begin{remark} 
In case (ii), the condition of nongenericity is given by the fact 
that all $\eps\in \Upsilon_R$  are real proportional to each other 
(i.e. they lie on the same real line in $\bC$ passing through the origin);

In case (iii) if one forces the next to the leading coefficient for some edge with integer 
slope and length of projection larger than $2$ to vanish, i.e. one forces the sum of 
the respective $\eps$ to be equal to $0$, then the conclusion of Corollary~\ref{cor:internal} will be valid
for a generic choice of the remaining coefficients at the vertices belonging to this edge.  

If the convex hull of $\mathcal U^\infty_R (w)$ does not tend to $\bC$, but $i_{n}>0$ which 
means that  $\mathcal U_R(w)\setminus \mathcal U^\infty_R (w)$ is nonempty, then the 
convex hull of  $\mathcal U_R(w)$ will tend to the convex cone with apex at $0$  spanned by the elements of  $\Upsilon_R$. 
\end{remark}

\begin{proof}[Proof of Proposition~\ref{pr:imp}]
By Corollary~\ref{cor:internal} we need to prove that the convex hull of $\Upsilon_R$ contains $0$ as its interior point

\begin{enumerate}[label={\upshape(\roman*)}]
 \item 
for any $R\in Pol(\NE)$ if $\NE$ is defining;

 \item 
 for generic $R\in Pol(\NE)$ if $\NE$ is almost defining;
 
 \item 
 if $\NE$ is non-defining, polynomials  $R\in Pol(\NE)$ for which $\Upsilon_R$ 
 contains $0$ as an interior point form a full-dimensional set  with the full-dimensional complement. 
\end{enumerate} 
 
 Indeed, assume that $\NE$ is defining. Then it contains an edge $e_s$ with the slope $-\al/\beta$ where $\alpha$ 
 and $\beta$ are coprime positive integers and $\beta \ge 3$. 
 Take any polynomial $R(u,v)\in Pol(\NE)$ and denote by  $R_s(u,v)$ the restriction of $R$ to $e_s$. 
 Substituting $u=\eps v^{\al/\beta}$ in the equation $R_s(u,v)=0$ and factoring out a power of $v$, 
 we get a univariate algebraic  equation for  $\eps$ which only involves powers of $\eps$ 
 which are multiples of $b\ge 3$. Since every non-vanishing $\eps$ appears in $\Upsilon_R$ 
 together with all $\eps\cdot e^{2\pi \sqrt{-1}\ell/b}$ for $\ell=1,\dots, b-1$ 
 one obtains that $0$ lies in the interior of the convex hull of $\Upsilon_R$.   
 
 Assume now that $\NE$ is almost defining. Then it either contains an edge $e_s$ with the slope $-\al/2$ and 
 length greater than $2$ or two edges with half integer slopes and length $2$ each. 
 (All the remaining edges have integer slopes.) In the former case, the algebraic equation satisfied by 
 $\epsilon$ has an even degree exceeding $2$ and contains only even powers of $\eps$. 
 Its non-vanishing solutions come in pairs of numbers of the form $(\al,-\al)$. 
 If at least two such pairs are non-proportional over $\bR$ (which happens generically) 
 then $0$ is the inner point of $\Upsilon_R$. Similarly, in the latter case we have two 
 second order equations without linear terms defining $\eps$. 
 Again typically their pairs of solutions are non-proportional over $\bR$ and the result follows. 
 
 Finally, assume that $\NE$ is non-defining. Then all edges, but possibly one have integer 
 slopes which means that the corresponding equations for $\eps$ will have all 
 possible monomials present and their non-trivial roots can either contain $0$ 
 inside their convex hull or lie in a half-plane of $\bC$ bounded by a 
 real line passing through the origin in which case $0$ is outside  (on the boundary of) this convex hull. If there is  one edge of length $2$ and 
 half-integer slope in $\NE_R$,  then it produces one pair of opposite values for $\eps$.  
\end{proof}

\section{Application of algebraic results to  invariant sets of degenerate operators}\label{sec:deg}

In what follows, we need to consider the action of  $T=\sum_{j=0}^k Q_k(x)\frac{d^j}{dx^j}$  on polynomials of the form $(x-t)^n$ for sufficiently large $n$. 
One has 
\[
T(x-t)^n=(x-t)^{n-k}\sum_{j=0}^k (n)_j(x-t)^{k-j}Q_j(x)=(x-t)^{n-k}\psi_T(x,n, t)
\]
where $\psi_T(x,n,t)$ is a trivariate polynomial. The important circumstance is that 
the essential part $\psi^+_T(x,n)$ of $\psi_T(x,n,t)$ is independent of $t$, see beginning of \S~\ref{sec:trop}. 
We will apply to $\psi_T^+(x,n)$ the results of the previous section and discuss how its zeros w.r.t $x$ 
behave when $n\to +\infty$.  Denote by $a_j x^{d_j}$ the leading monomial of $Q_j(x)$ and consider the polynomial 
\[
\widetilde \psi_T(x,n)=\sum_{j=0}^ka_jn^jx^{d_j+k-j}.
\]
(It contains much fewer monomials than $\psi_T(x,n,t)$, but with exactly the same coefficients.) 
Notice that  the essential part $\psi^+_T(x,n)$ is obtained from $\widetilde \psi_T(x,n)$ by removing those monomials 
which do not belong to $\NE(\psi_T)$. 

Taking the symbol polynomial $G_T(x,y)=\sum_{j=0}^kQ_k(x)y^j$ of  $T$, we 
introduce its truncation $\widetilde G_T(x,y)=\sum_{j=0}^ka_jy^jx^{d_j}$ and observe that $\widetilde \psi_T(x,n)$ is 
obtained from $\widetilde G_T(x,y)$ by substituting $y$ by $n$ and adding $k-j$ to the powers $d_j$  of $x$ 
 of the respective monomial. Thus  the Newton polygon of $\widetilde \psi_T(x,n)$ is 
obtained from the Newton polygon of  $\widetilde G_T(x,y)$  by the affine transformation $A$ sending $(i,j)$ to 
$(i+k-j,j)$. Therefore $\NE(\psi_T)$ is obtained from the part of the 
boundary of the Newton polygon of $\widetilde \psi_T(x,y)$  under the latter affine transformation,
see Fig.~\ref{fig:neBoundaryShift} for an example.

\begin{figure}[!ht]
\[
\begin{tikzpicture}[baseline=(current bounding box.center)]
\draw[step=1em, lightgray, very thin] (0,0) grid (9em, 9em);
\draw[x=1em,y=1em,black,->] (0,0)--(9,0);
\draw[x=1em,y=1em,black,->] (0,0)--(0,9);
\begin{scope}
\draw[black, very thick,x=1em,y=1em] (3,7) -- (6,6)--(7,2);
\draw[black, dashed,x=1em,y=1em] (7,2)--(3,0)--(1,1)--(0,5)--(3,7);
\node[x=1em,y=1em] at (3,4.0) {$N_T$};
\node[x=1em,y=1em] at (0,9.5) {y};
\node[x=1em,y=1em] at (9.5,0) {x};
\end{scope}
\end{tikzpicture}
\overset{A}\longrightarrow
\begin{tikzpicture}[baseline=(current bounding box.center)]
\draw[step=1em, lightgray, very thin] (0,0) grid (12em, 9em);
\draw[x=1em,y=1em,black,->] (0,0)--(12,0);
\draw[x=1em,y=1em,black,->] (0,0)--(0,9);
\begin{scope}
\draw[black, very thick,x=1em,y=1em] (3,7) -- (7,6)--(12,2);
\draw[black, dashed,x=1em,y=1em] (12,2)--(10,0)--(7,1)--(2,5)--(3,7);
\node[x=1em,y=1em] at (6.5,4.0) {$N_\psi$};
\node[x=1em,y=1em] at (0,9.5) {n};
\node[x=1em,y=1em] at (12.5,0) {x};
\end{scope}
\end{tikzpicture}
\]
 \caption{The affine transformation $A$ sending $N_T$ to $N_\psi$. 
 Here $T=(x^3+\dots)\frac{d^7}{dx^7}+(x^6+\dotsb)\frac{d^6}{dx^6}+\frac{d^5}{dx^5}+(x^7+\dotsb)\frac{d^2}{dx^2}+(x+\dotsb)\frac{d}{dx}+(x^3+\dotsb)$,
 $\widetilde G_T(x,y)=x^3y^7+x^6y^6+y^5+x^7y^2+xy+x^3$ and $\widetilde \psi_T(x,n)=n^7x^3+n^6x^7+n^5x^2+n^2x^{12}+x^{10}.$}\label{fig:neBoundaryShift}
\end{figure}
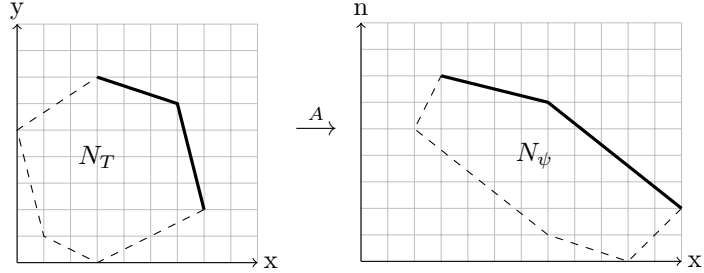

Denote the Newton polygon of $\widetilde G_T(x,y)$ by $N_T$ and the Newton polygon of  $\widetilde \psi_T(x,y)$ by $N_\psi$. 
We have that $N_\psi=A\circ N_T$.   The relation between the slopes of edges before and after the affine transformation $A$ is as follows.  

    If the slope $sl$ of an edge of $N_T$ equals $sl=\frac{\mu}{\nu}$ where $\mu$ and $\nu$ are coprime  integers and $\nu>0$, 
    then the slope of its image denote by $asl$ is given by $asl=\frac{\mu}{\nu-\mu}$ which implies that $asl=\frac{sl}{1-sl}$ 
    or, equivalently, $sl=\frac{asl}{1+asl}$. Therefore if $asl$ is a negative integer  then we get 
    \[
    asl=-J, J>0 \Leftrightarrow sl=\frac{J}{J-1}.
    \]
    Obviously any $sl$ of the above form is positive (or $+\infty$). 
    
    Analogously, if $asl$ is a negative half-integer  then we get 
    \[
    asl=-\frac{J}{2}, J>0 \text{ and odd} \Leftrightarrow sl=\frac{J}{J-2}.
    \]
    Again any $sl$ of the above form is positive  with the only exception $J=1$ for which $sl=-1$. 

It is easy to describe $A^{-1}(\NE_\psi)$ as the part of the boundary $N_T$ starting at $V_{n}$ and 
going southeast till we either reach the lowest point of the polygon or till the slope of 
the next edge  becomes smaller than or equal to $1$. Denote $A^{-1}(\NE_\psi)$ as $\mathfrak B_T$ 
and call it the \defin{shifted northeastern border} of $N_T$. 

One can easily check that for $T=\sum_{j=0}^k Q_k(x)\frac{d^j}{dx^j}$, 
the corresponding $\NE(\psi_T)$ is a single point if and only if $T$ is non-degenerate. 
So for any degenerate $T$, its $\NE(\psi_T)$ contains at least one edge. 
 Additionally, $asl<0$ if and only if $\frac{1}{sl}<1$ which means that either $sl<0$ or $sl>1$. 

Observe that the vertex $V_{n}$ of $\widetilde \psi$ coincides with that of $\widetilde G$. 
 The following notion is important for the rest of the paper. 

\begin{definition} A degenerate operator $T$ is called \textcolor{blue}{defining/almost defining/non-defining} if its Newton polygon $N_\psi$ is defining/almost defining/non-defining resp., see Definition~\ref{def:newtonCases}. 
In terms of the Newton polygon $N_T$ this means that its shifted northeastern border $\mathfrak B_N$ is not a single point and 
in the \emph{defining} case it contains an edge with the slope of the form $\frac{J}{J-\beta}$ with $\beta\geq 3$, 
in the \emph{almost defining} case all edges  of  $\mathfrak B_N$ have slopes $\frac{J}{J-1}$ but there 
exists either one edge with slope  $\frac{J}{J-2}$, $J$ odd and length greater than $2$  or two such 
edges with length $2$; and in the \emph{non-defining} case contains  edges of arbitrary integer 
length with slopes $\frac{J}{J-1}$,  $J$ being a positive integer, except for possibly one edge of 
integer length $2$ whose slope is $\frac{J}{J-2}$, $J$ odd. 
\end{definition}

The following result is an easy consequence of our previous considerations. 

\begin{theorem}\label{th:defining}
For any nonnegative integer $n$ and (almost) any degenerate 
operator $T$ whose $N_T$ is (almost) defining, the only set contained in $\invset{\geq n}$ is $\bC$. 
\end{theorem} 


\subsection{Degenerate operators with non-defining Newton polygons}\label{sec:nondefining} 

As we have seen above the convex hull of the set $\Upsilon_T$ of all leading constants for (almost) every degenerate $T$ with (almost) defining $N_T$  contains $0$ as its interior point. 

For degenerate $T$ with non-defining $N_T$ whose northeastern border we will denote by $\NE_T$, it might still happen that $0$ is the interior point of the latter convex hull in which case the conclusion of Theorem~\ref{th:defining} holds.  However for a full-dimensional subset of $Pol(\NE)$  with a given non-defining $\NE$, their leading constants belong some half-plane in $\bC$ bounded by a line passing through $0$ and therefore $0$ lies  on the boundary of their convex hull. In this situation   the conclusion of Theorem~\ref{th:defining} fails and we will discuss this case below.

\begin{definition} 
Given a finite set  $\mathcal U=\{u_1,\dotsc, u_k\}$ of (not necessarily distinct) complex numbers,
we define the cone $\C^+ \mathcal U \subseteq \bC$ generated by $\mathcal U$ as  given by 
\[
\C^+ \mathcal U\coloneqq \{\al_1 u_1+\al_2 u_2 +\dotsb + \al_\ell u_\ell\}, \text{where}\; \al_j\ge 0, \:j=1,\dotsc, \ell.
\]

 We say that a set $S\subseteq \bC$ is 
\defin{closed with respect to  $\C^+ \mathcal U\subseteq \bC$} if for any complex number $z\in S$ 
and any $v\in \C^+ \mathcal U$,  $z+v$ belongs to $S$. 
\end{definition}
Obviously, $0$ is the interior point of the convex hull of $\mathcal U=\{u_1,\dotsc, u_\ell\}$ if and only if  $\C^+ \mathcal U=\bC$. 

Given a degenerate operator $T$ with non-defining polygon $N_T$, denote by $\Upsilon_T \coloneqq \{\eps_1,\dotsc, \eps_m\}$ 
 the collection of all its leading constants and set $\C^+_T:=\C^+(\Upsilon_T)$. 
As we mentioned above, if  $\C^+_T=\bC$, then the conclusion of Theorem~\ref{th:defining} holds.
Let us assume now that $\C^+_T$ is a closed sector in the plane with positive angle $\le \pi$. 
(We are then  missing two remaining cases: $\C^+_T$ being a line through  the origin and $\C^+_T$ being a half-line through  the origin.) 

\begin{remark} Recall that by item (2) of Theorem~\ref{th:generalGeN},  any set $S\in \invset{\ge n}$ is unbounded and belongs to $\invset{\ge 0}$, i.e. is unbounded and $T$-invariant. 
\end{remark} 

\begin{lemma}\label{lm:cone}
In the above notation, any  $T$-invariant set $S$ is closed with respect to $\C^+_T$. 
\end{lemma} 
\begin{proof}  Indeed, take a point $t\in S$ and consider the sequence of polynomials $T(x-t)^n$ when $n$ increases. 
For $n\to \infty$, the roots of $T(x-t)^n$ whose absolute values tend to infinity will be spreading out to infinity 
 approaching some rays whose directions are given by the elements of   $\Upsilon_T$. 
Since every $S$ must be convex the result follows. 
\end{proof} 

\begin{corollary}\label{th:cone} In the above notation, if the product of the leading coefficient $Q_k(x)$ and the constant term $Q_0(x)$ 
of the operator $T$ is not a constant, then any $T$-invariant set $S$ contains the the Minkowski sum $Conv(Q_kQ_0)\oplus \C^+_T\subset  \bC$ of   $\C^+_T$ and  $Conv(Q_kQ_0)$;  the latter set being  the convex hull of the union of all roots of $Q_k(x)$ and $Q_0(x)$.
\end{corollary}

\begin{proof} It is an easy to see if any $T$-invariant set $S$ must contain the zero locus of $Q_k(x)$ as well as of $Q_0(x)$ which by convexity of $S$ implies that it should contain $Conv(Q_kQ_0)$. Applying Lemma~\ref{lm:cone} we get the required result.
\end{proof}

Let us now present some conditions for the existence of non-trivial  $T$-invariant set for a degenerate operators  $T$.

\subsection{Degenerate operators with non-defining Newton polygon and constant leading term}

 The remaining case of a constant leading term is discussed below. 
One can easily check that  the class of degenerate operators 
\[
T = \frac{d^k}{dx^k}+Q_{k-1}(x)\frac{d^{k-1}}{dx^{k-1}} + \dotsb + Q_k(x)
\]
with  non-defining $N_T$ splits into two subclasses:

\begin{enumerate}[label={\upshape\Alph*:},leftmargin=*]
\item 
  operators with constant coefficients; 

\item 
 operators satisfying the following three conditions: 
 
\begin{enumerate}[label={\upshape(\roman*)}]
 \item 
  $\deg Q_{k-1}= 1$;  

  \item 
 $\deg Q_j\le 1$ for $j=0,\dotsc, k-2$;  
 
 \item 
  if $j_{min}$ is the smallest  value of $j$ for which $\deg Q_j=1$, then $Q_\ell$  must vanish for all $\ell\le j_{min}-2$.  
\end{enumerate}
\end{enumerate}

For the more interesting subclass (B) the northeastern border of such operator $T$ can consist of 1, 2 or three edges, 
see Figure~\ref{fig:neCases} below. If it consists of $1$ edge then after an affine change of $x$ we can reduce such an operator to
\[
T=\frac{d^k}{dx^k}-x\frac{d^{k-1}}{dx^{k-1}}+\al \frac{d^{k-2}}{dx^{k-2}},\; \al \in \bC.
\]

 If it consists of $2$ edges then after an affine change of $x$ we can reduce such an operator to 
\[
T=\frac{d^k}{dx^k}-x\left(\frac{d^{k-1}}{dx^{k-1}}+\sum_{i=1}^\ell \al_i  \frac{d^{k-1-i}}{dx^{k-1-i}}\right)+\sum_{i=1}^\ell \beta_i  \frac{d^{k-2-i}}{dx^{k-2-i}} 
\]
where $\ell\le k-1$ is a positive integer and all $\al_i$ and $\beta_i$ are arbitrary complex numbers with the only restriction $\al_\ell\neq 0$. 

Finally, if  it consists of $3$ edges then after an affine change of $x$ we can reduce such an operator to 
\[
T=\frac{d^k}{dx^k}-x\left(\frac{d^{k-1}}{dx^{k-1}}+\sum_{i=1}^\ell \al_i  \frac{d^{k-1-i}}{dx^{k-1-i}}\right)+
\sum_{i=1}^\ell \beta_i  \frac{d^{k-2-i}}{dx^{k-2-i}}+\beta_{\ell+1}\frac{d^{k-3-\ell}}{dx^{k-3-\ell}}
\]
where $\ell\le k-3$ is a positive integer, all $\al_i$ and $\beta_i$ are 
arbitrary complex numbers with the  restrictions $\al_\ell\neq 0$ and $\beta_{\ell+1}\neq 0$. 
We will discuss these subcases below.

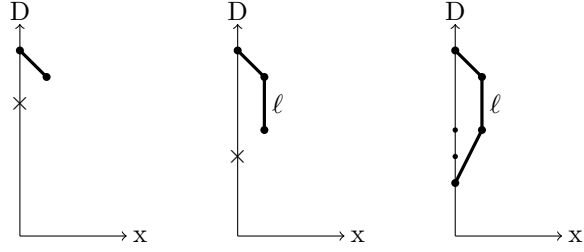
\begin{figure}[!ht]
\begin{tikzpicture}[baseline=(current bounding box.center)]
\draw[x=1em,y=1em,black,->] (0,0)--(4,0);
\draw[x=1em,y=1em,black,->] (0,0)--(0,8);
\draw[black, very thick,x=1em,y=1em] (0,7)--(1,6);
\fill[black,x=1em,y=1em] (0,7) circle (0.15em);
\fill[black,x=1em,y=1em] (1,6) circle (0.15em);
\node[x=1em,y=1em] at    (0,5) {$\times$};
\node[x=1em,y=1em] at (0,8.5) {D};
\node[x=1em,y=1em] at (4.5,0) {x};
\end{tikzpicture}
\qquad
\begin{tikzpicture}[baseline=(current bounding box.center)]
\draw[x=1em,y=1em,black,->] (0,0)--(4,0);
\draw[x=1em,y=1em,black,->] (0,0)--(0,8);
\draw[black, very thick,x=1em,y=1em] (0,7)--(1,6)--(1,4);
\fill[black,x=1em,y=1em] (0,7) circle (0.15em);
\fill[black,x=1em,y=1em] (1,6) circle (0.15em);
\fill[black,x=1em,y=1em] (1,4) circle (0.15em);
\node[x=1em,y=1em] at    (0,3) {$\times$};
\node[x=1em,y=1em] at    (1.5,5) {$\ell$};
\node[x=1em,y=1em] at (0,8.5) {D};
\node[x=1em,y=1em] at (4.5,0) {x};
\end{tikzpicture}
\qquad
\begin{tikzpicture}[baseline=(current bounding box.center)]
\draw[x=1em,y=1em,black,->] (0,0)--(4,0);
\draw[x=1em,y=1em,black,->] (0,0)--(0,8);
\draw[black, very thick,x=1em,y=1em] (0,7)--(1,6)--(1,4)--(0,2);
\fill[black,x=1em,y=1em] (0,7) circle (0.15em);
\fill[black,x=1em,y=1em] (1,6) circle (0.15em);
\fill[black,x=1em,y=1em] (1,4) circle (0.15em);
\fill[black,x=1em,y=1em] (0,2) circle (0.15em);
\fill[black,x=1em,y=1em] (0,3) circle (0.10em);
\fill[black,x=1em,y=1em] (0,4) circle (0.10em);
\node[x=1em,y=1em] at    (1.5,5) {$\ell$};
\node[x=1em,y=1em] at (0,8.5) {D};
\node[x=1em,y=1em] at (4.5,0) {x};
\end{tikzpicture}
\caption{$\NE$ borders of the three sub-cases in subclass (B).
 Here, $\times$ denotes a monomial that might be present, but all monomials below $\times$
 must be absent, i.e. have vanishing coefficients. 
 }\label{fig:neCases}
\end{figure}

\subsubsection{Subclass A, i.e., linear differential operators with constant coefficients}\label{sub:constant}
Observe that in the case case of constant coefficients, if $S$ is a $T$-invariant set, 
then for any $a\in \bC$, $S_a \coloneqq S+a$ is a $T$-invariant set as well. (Similarly for $T_{\ge n}$-invariant sets).

\begin{proposition} \label{pr:constantcoeff} 
Let 
\begin{equation}\label{eq:constant}
T=a_k\frac{d^k}{dx^k} + a_{k-1}\frac{d^{k-1}}{dx^{k-1}}+\dotsb +a_0,\; a_k\neq 0
\end{equation} 
 be a 
linear differential operator with constant coefficients. Let $\La_T^{-1}=\{\la_1^{-1},\dots, \la_k^{-1}\}$ 
be the set of the inverses of  characteristic exponents (not necessarily distinct), where  
\[
a_kt^k + a_{k-1}t^{k-1}+\dots +a_0=a_k(t-\la_1)(t-\la_2) \dotsm (t-\la_k).
\]
Then a convex set $S\subseteq \bC$ is $T$-invariant if and only if $S$ is closed with respect to $\C \La_T^{-1}$. 
\end{proposition}

\begin{remark} We use the convention that if $\la_j=0$ then its inverse disappears from the list $\La_T^{-1}$. 
Further notice that if $\C \La_T^{-1}=\bC$ which happens in the open (in the usual topology) subset  of 
linear differential operators of the form \eqref{eq:constant} of any 
given order $k\ge 3$, the only $T$-invariant $S\subseteq \bC$ is the whole $\bC$.  
\end{remark}

\begin{proof}[Proof of Proposition~\ref{pr:constantcoeff}]

To prove the implication $\Rightarrow$, 
we invoke Lemma~\ref{lm:cone} and the observation that $\C \La_T^{-1}=\C^+_T$.


To prove the converse implication, we proceed by induction on $k$ whose base is the following statement.  

\begin{lemma}\label{lm:order1}
 For an operator $T = \diffx - \la$, a convex set $S\subseteq \bC$ is $T$-invariant if 
 and only if for any $x\in S$ and  $\tau>0$, the number $x-\tau \la$ 
 belongs to $S$ which is equivalent to $S$ being closed with respect to $\C \La_T^{-1}$. 
 \end{lemma}
 \begin{proof}
In case $\la=0$, any convex set $S$ is $T$-invariant by the Gauss--Lucas theorem. 
For $\la\neq0$, using the rescaling of $x$ we can reduce $T$ to the special case $\diffx+1$. 
Observe that for any polynomial $p(x)$, the zeros of $p^\prime(x) + p(x)$ coincide with that of $e^{-x}(p(x) e^x)^\prime$. 
Recall that $e^x=\lim_{n\to \infty} (1+\frac{x}{n})^n=\lim_{n\to \infty}\frac{(x+n)^n}{n^n}$. 
By translation invariance, we can additionally assume that either all roots of $p(x)$ are real or 
among these roots there is at least one with a positive imaginary part and at least one with 
the negative imaginary part. For any natural $n$, all roots of $((x+n)^np(x))^\prime$ lie in the 
convex hull of all roots of $p$ appended with $-n$. When $n\to \infty$, we get the 
required statement. In other words, all roots of $p^\prime(x)+p(x)=e^{-x}(p(x)e^x)^\prime$ lie in 
the infinite polygon (or half-line) formed by the parallel translation of  the convex hull of all roots of $p$ to infinity in the direction $-1$.
 \end{proof}

To continue our proof by induction, notice that the  operator \eqref{eq:constant} factorises as 
\[
T=a_k \left( \diffx-\la_1 \right) \left(\diffx-\la_2 \right) \dotsm \left(\diffx-\la_k \right)
 =\left( \diffx-\la_1 \right)\widetilde T,
\]
where $\widetilde T$ has order $k-1$. Observe that the factors in the above expansion commute. 
By inductive hypothesis, $S$ is a $\widetilde T$-invariant subset if and only if it is closed  with respect to $\C \La_{\widetilde T}^{-1}$.

Assume that $S\subset \bC$ is closed with respect to $\C \La_T^{-1}$ and let $p(x)$ 
be a polynomial with all roots in $S$. We need to show that $T(p)$ has all roots in $S$. We have that 
$T(p)=\left( \diffx-\la_1 \right)\widetilde T(p)$. Since  $S\subset \bC$ is closed
with respect to $\C \La_T^{-1}$ it is also closed with respect 
to  $\C \La_{\widetilde T}^{-1}$ which implies that all roots of $\widetilde T(p)$ lie in $S$. 
Using Lemma~\ref{lm:order1} again and the fact that $\C \La_T^{-1}$ contains $\C \la_1^{-1}$ we
get that all roots of $T(p)=\left( \diffx-\la_1 \right)\widetilde T(p)$ lie in $S$ as well. 
  \end{proof}

\subsubsection{Subclass B, i.e., operators with constant leading term and $\deg Q_{k-1}=1$} 
In this case  we currently  have only a number of sporadic results. 

\smallskip
Let us  start with operators of order $1$. 
After an affine change of $x$,  
we  only need to consider one single operator $T=\diffx-x$. 
The following statement holds. 

\begin{lemma}\label{lm:d+x} 
For $T=\diffx-x$, its minimal $T$-invariant set $\minvset{\ge 0}$  is the real axis.
\end{lemma}
\begin{proof} 
It is easy to check using  \cite[Theorem 1.3]{BorceaBranden2010} that $T$ is a \textcolor{blue}{\emph{hyperbolicity preserver}}, i.e., $T$ sends every real-rooted 
polynomial to a real-rooted polynomial (or $0$).   

Recall that the \textcolor{blue}{\emph{symbol}} $F_T(x,y)$ of the differential operator $T = \sum_{j=0}^k Q_j(x)\diff{j}$ is by definition given by 
$F_T(x,y):=\sum_{j=0}^k Q_j(x)y^j.$  The above mentioned criterion claims that $T$ is a hyperbolicity preserver if and only if the real algebraic \textcolor{blue}{\emph{ symbol curve}}  $\Gamma_F\subset \bR^2$  given by $F_T(x,y)=0$ must intersect each affine line with negative slope in all real points. (The real plane $\bR^2$ is equipped with coordinates  $(x,y)$).  In other words,  this number of real intersection points counting multuplicity must be equal to the degree of $F_T(x,y)$. In the case under consideration, 
 the symbol $F_T(x,y)$ of $T=\diffx-x$ equals $y-x$ and 
its symbol curve has one real intersection point with each real affine line except for those parallel to $x=y$.   
One can also check that no subinterval of $\bR$ is a $T$-invariant set. Indeed, applying $T$ to $x-\alpha$, $\alpha\in \bR$ we get
\[
T(x-\alpha)=-x(x-\alpha)+1=-(x^2-\alpha x-1)
\]
whose roots are $\alpha/2\pm \sqrt{(\alpha/2)^2+1}$. 
They are the endpoints of a real interval containing $\alpha$.
\end{proof} 

\smallskip
The next results describe which operators $T$ belonging to  the class  B preserve a given half-plane in $\bC$.  As a consequence we  characterize   hyperbolicity preserving $T$ in this class. 


Observe that for any operator $T$ belonging to the class $B$, its  symbol $F_T(x,y)$ is of the form $U(y)-xV(y)$ 
where $U(y)=y^k+\dotsb$ and $V(y)=y^{k-1}+\dotsb$. (Here $\dotsb$ stands for lower degree terms in $y$).  

\begin{lemma}\label{Hstab}
Let $H \subset \bC$ be an open half-plane represented as 
\[
H= \{ az +b : \Im z \leq 0\},
\]
where $a, b \in \bC$ and $a\neq 0$, and let 
\[
T= U\left(\frac d {dx}\right)+ xV\left(\frac d {dx}\right),
\]
where $U$ and $V$ are polynomials.  Then the following are equivalent 
\begin{enumerate}
\item $H$ is $T_n$-invariant for all $n$, 
\item The bivariate polynomial $U(-y/a)+bV(-y/a)+aV(-y/a) x$ is stable in $(x,y)$, 
\item Either $V \equiv 0$ and $U(-y/a)$ is stable, or the rational map 
\[
z \longmapsto \frac 1 a \frac {U(-z/a)}{V(-z/a)} + \frac b a, 
\]
maps the open upper half-plane to the closed upper half-plane. 
\end{enumerate}
\end{lemma}

(For the notions of stability and $T_n$-invariance see  Definitions~\ref{def:T_n} and\ref{def:stablePoly}).

\begin{proof}
By Proposition~\ref{prop:halfPlaneCharaterization}, the first two statements are equivalent. The polynomial in (2) is stable if whenever $z$ is in the upper half-plane and 
\[
U(-z/a)+bV(-z/a)+aV(-y/a) w =0,
\]
then $w$ is in the closed lower half-plane. Solving for $w$ gives the equivalence of (2) and (3). 
\end{proof}

We recall the following version of the Hermite-Biehler Theorem  from \cite{LYPSI}. 
\begin{lemma}\label{HB}
Let $f, g \in \bR[x]$. The following are equivalent
\begin{itemize}
\item the univariate polynomial $f(x)+ig(x)$ is stable,
\item the bivariate polynomial $f(x)+yg(x)$ is stable, 
\item $f$ and $g$ are real-rooted, their zeros interlace, and 
\[
W(f,g)= f'(x)g(x)-f(x)g'(x) \geq 0, \ \ \ \text{ for all } x \in \bR.
\]
\end{itemize}
\end{lemma}
Also if the zeros of $f$ and $g$ interlace, then either $W(f,g) \geq 0$ for all $x$ or $W(g,f) \geq 0$ for all $x$. 
\begin{corollary}
Let 
\[
T= U\left(\frac d {dx}\right)+ xV\left(\frac d {dx}\right),
\]
where $U(y), V(y) \in \bR[y]$. Then $\bR$ is $T_n$-invariant for all $n$ if and only if 
\begin{itemize}
\item there is a nonzero constant $\xi \in \bC$ such that $\xi U(y), \xi V(y) \in \bR[y]$, and 
\item the zeros of $\xi U(y)$ and $\xi V(y)$ are real and interlacing, and 
\[
W(\xi V(y), \xi U(y)) \geq 0, \ \ \ \ \ \text{ for all } x \in \bR.
\]
\end{itemize}
\end{corollary}

\begin{proof}
If $\bR$ is $T_n$-invariant for all $n$, then there is a nonzero $\xi \in \bC$ such that 
$\xi T : \bR[x] \to \bR[x]$, see Section 4 of \cite{LYPSI}. 

Moreover for any differential operator $T$ with real coefficients, $\bR$ is $T_n$-invariant for all $n$ if and only if the closed lower half-plane is $T_n$-invariant for all $n$, see Theorem 1.2 and Theorem 1.3 in \cite{BorceaBranden2010}. Hence the result follows from Lemma~\ref{Hstab} and Lemma~\ref{HB}.
\end{proof}

\section{Variations of the original set-up}\label{sec:examples}

Above we have mainly concentrated on  invariant sets for  roots of polynomials of degree at least $n$.
Currently  we neither have a 
description of the minimal invariant sets whose existence we have established nor   
a  numerically stable procedure which will construct them or their approximations in specific examples. 

The goal of this section is to present some interesting variations 
of our basic notion of invariant sets together with numerical 
examples illustrating the other types of invariant sets introduced below.  These notions are of independent interest and might be easier to study.


\medskip 
\noindent 
\textbf{Variation 1: invariant sets for roots of polynomials of a  fixed degree.}
Instead of looking for  a set which  is invariant for roots of polynomials of degree at least $n$,
we can relax the requirement and ask  that a set is only invariant for roots of polynomials of degree exactly $n$. We call this property $T_n$-invariance, see  Definition~\ref{def:T_n}.  

Given $T$ and $n$, we denote by $\invset{n}$ the family of $T_n$-invariant sets   and we denote by $\minvset{n}$ the corresponding unique minimal closed invariant 
set (if it exists), see Introduction. 
Note that $\minvset{n} \subseteq \minvset{\geq n}$. It is natural to study $\minvset{n}$ for  exactly solvable operators $T$ since in this case they preserves the degrees of polynomials they act upon. 

One can observe that in many cases $\minvset{n}$ can have a complicated structure --- in particular,
it does not need to be convex, and it can be a fractal etc.  An illustration can be found in Example~\ref{ex:juliaExample} and  Figure~\ref{fig:julia}. We plan to carry out the detailed study of  $T_n$-invariant sets   in the sequel paper \cite{AlBrSh2}.

\begin{example}\label{ex:juliaExample}
The minimal invariant set $\minvset{1}$ for the differential operator $T = (x^2-x+i)\diffx + 1$
coincides with the classical Julia set associated with $f(x) = x^2+i$.
\begin{figure}[!ht]
\centering
  \includegraphics[width=0.5\textwidth]{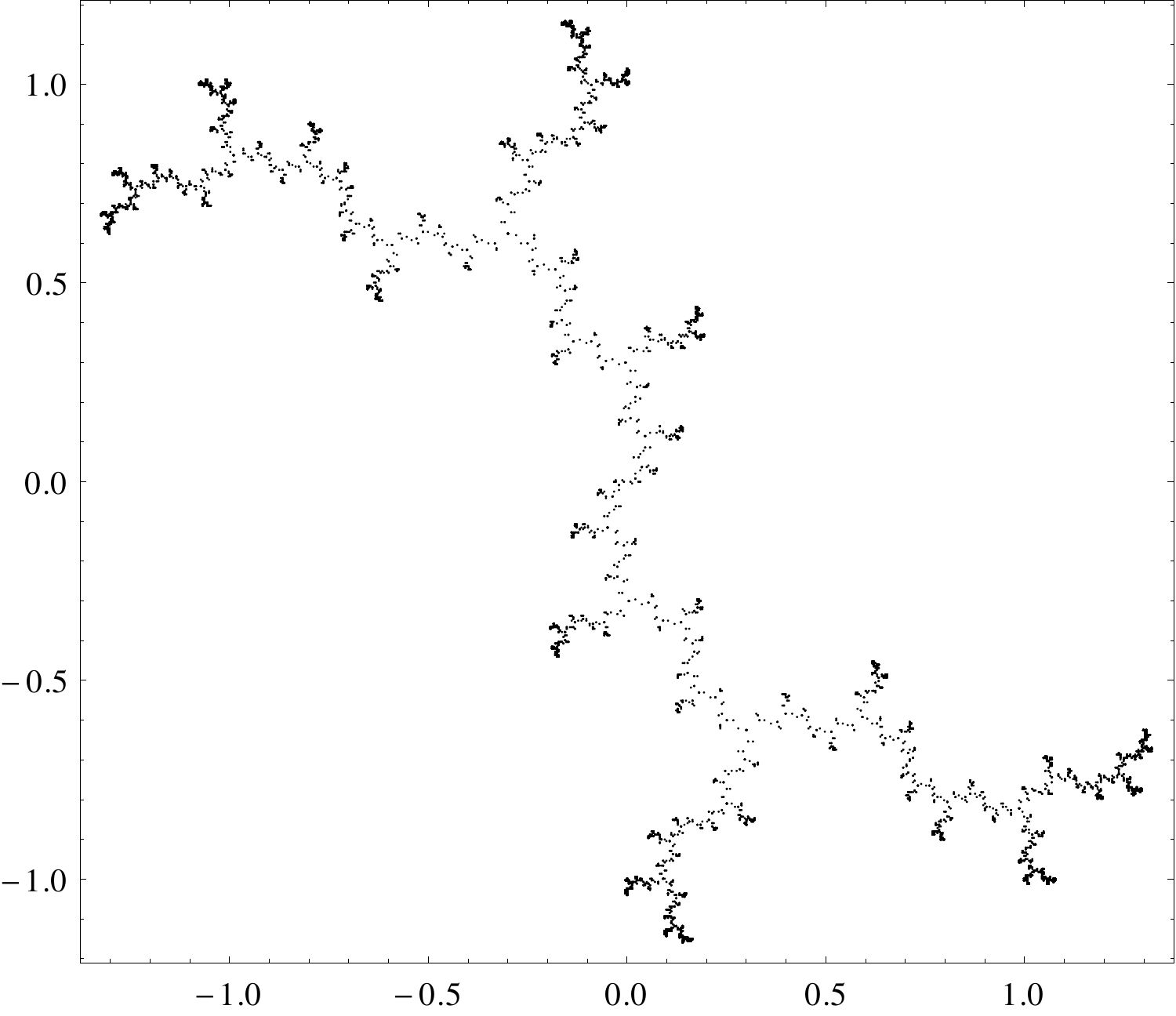}
\caption{The minimal set  $\minvset{1}$ for the operator $(x^2-x+i)\diffx + 1$.
This set has the property that if $t \in \minvset{1}$, then $\pm \sqrt{t-i}$ is also in $\minvset{1}$.
}\label{fig:julia}
\end{figure}
\end{example}

\medskip 
\noindent 
\textbf{Variation 2: Hutchinson-invariant sets.}
A set $S \subset \bC$ is called \defin{Hutchinson-invariant in degree $n$} 
if every polynomial of the form $P(x) = (x-t)^n$ with $t\in S$,
has the property that $T(P)$ has all roots in $S$ (or is constant).
In particular,   a $T_1$-invariant set is a   Hutchinson-invariant set in degree $1$ and vice versa. However, for $n>1$, $T_n$-invariant sets and  
Hutchinson-invariant sets  in degree $n$ in general do not  coincide.  
We denote by $\hinvset{n}$ the collection of all Hutchinson-invariant sets in degree $n$ and by $\hminvset{n}\in \hinvset{n}$ the unique minimal under inclusion 
closed Hutchinson-invariant set in degree $n$ (if it exists). 
Notice that
\[
  \hminvset{n} \subseteq \minvset{n} \subseteq \minvset{\geq n}.
\]
In particular, if $\hminvset{n}$ exists, then $\minvset{n}$ and $\minvset{\geq n}$  exist as well.

To explain our choice of terminology, recall that a \textcolor{blue}{ \emph{Hutchinson operator}} is defined by a finite collection  of univariate functions $\phi_1,\dotsc,\phi_m$ and its invariant sets were introduced and studied in \cite{Hutchinson1981} as well as  a large number of follow-up papers.  
In our situation, let us assume that the action of  $T$ on $(x-t)^n$ factorizes as:
\begin{equation}\label{eq:factorizable}
 T((x-t)^n) = \left(x-(a_1t+b_1)\right)\dotsm \left(x-(a_m t+b_m)\right),
\end{equation}
see e.g. \eqref{eq:exactlySovlableNonCompact}. 
Then we have that if $S\subset \bC$ is Hutchinson-invariant in degree $n$,
then $f_i(S) \subseteq S$ for all $i=1,2,\dotsc,m$, where $f_i(t)=a_i t+ b_i$.
If all these $f_i$ are contractions, that is, $|a_i|<1$, one can show that there is a unique minimal non-empty closed
Hutchinson-invariant set $S$, and it is exactly the invariant set associated with the 
 \textcolor{blue}{\emph {Hutchinson operator}} defined by $f_1,\dotsc,f_m$, see \cite{Hutchinson1981}. (One can also consider other types of factorizations similar to \eqref{eq:factorizable} with e.g. polynomial or rational factors.) 
This observation implies that one can obtain many classical fractal sets such as the Sierpinski triangle, the Cantor set,
the L\'evy curve  and the Koch snowflake as Hutchinson-invariant sets, see Example~\ref{ex:levy}.
In particular, $\minvset{n}$ does not have to be connected.

Julia sets associated with rational functions can also be 
realized as Hutchinson-invariant sets of appropriately chosen operators $T$, see \cite{AlBrSh2}.
Let us  illustrate the situation with  
Example~\ref{ex:levy} and Example~\ref{ex:juliaSpecial}.

\begin{example}\label{ex:levy}
For the differential operator $T = x(x+1)\diff{2} + i \diffx + 2$, the set $\hminvset{2}$
is a L\'evy curve. The roots of $T((x-t)^2)$ are given by 
\[
x= \frac{1+i}{2}t \qquad \text{ and }\qquad x=\frac{1-i}{2}(t - i).
\]
The two maps 
\begin{equation}\label{eq:levyMaps}
 t \mapsto \frac{1+i}{2}t \qquad \text{ and } \qquad t \mapsto \frac{1-i}{2}(t - i)
\end{equation}
are both affine contractions which together produce a fractal L{\'e}vy curve as their invariant set, see Figure~\ref{fig:levy}.
In particular, every member of $\invset{2}$
must contain $\hminvset{2}$ given by the latter curve which also implies that $\minvset{2}$ exists.
\begin{figure}[!ht]
 \centering
  \includegraphics[trim=100 330 100 220,clip,width=0.6\textwidth]{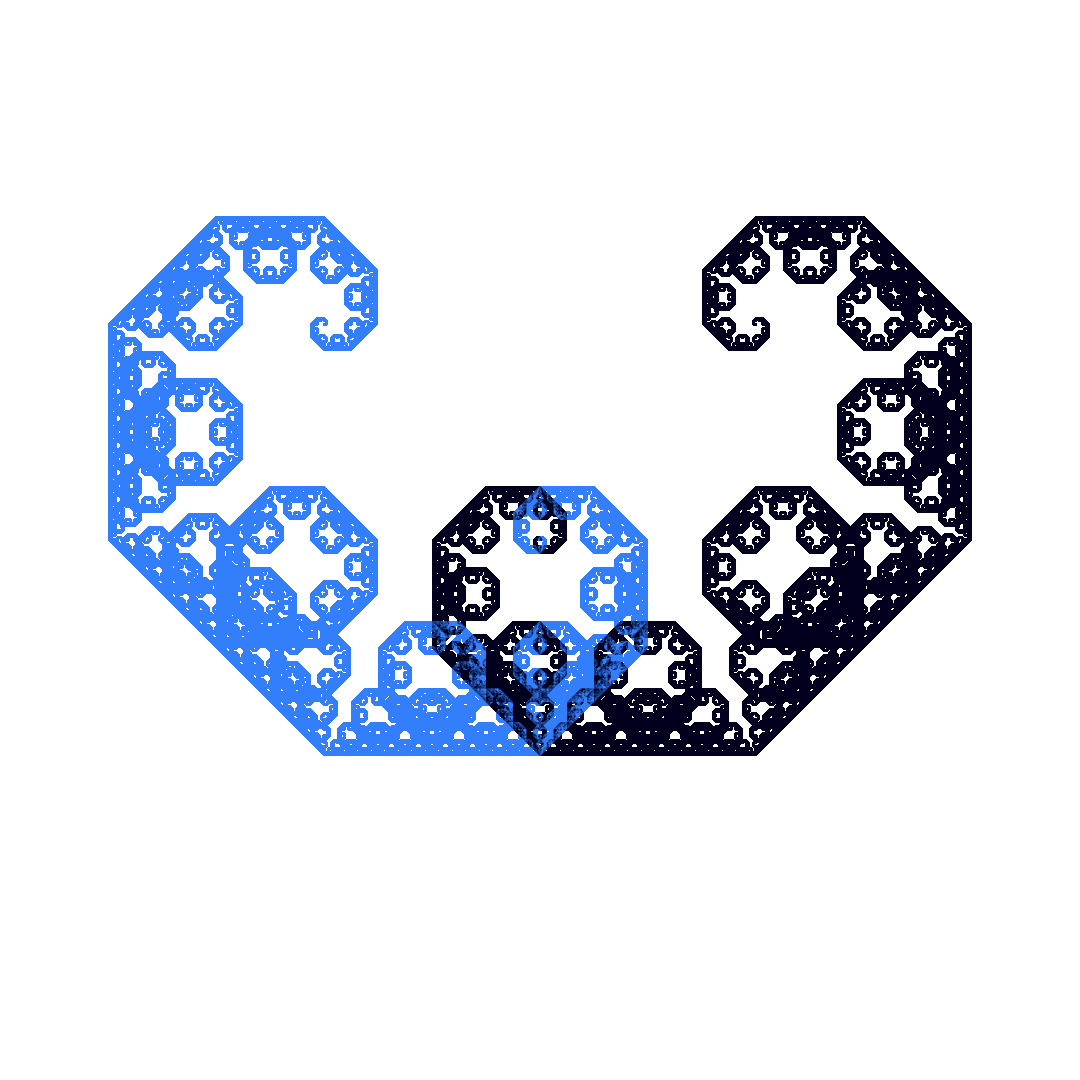}
  \caption{
The Hutchinson-invariant set $\hminvset{2}$ 
for the operator $T = x(x+1)\diff{2} + i \diffx + 2$.
The two colors indicate the image of the set under the two maps in \eqref{eq:levyMaps}.
}\label{fig:levy}
 \end{figure}

 \begin{example}\label{ex:juliaSpecial}
The differential operator $T = x(x-1)\diffx + 1$
admits \emph{two} minimal\footnote{Minimal here means that no proper closed subset is an invariant set.} 
sets $\hminvset{1}$, one of which is the one-point set $\{0\}$ and the other is the unit circle.
This fact is in line with known properties of the Julia sets; some very special rational functions admit
several completely invariant sets containing one or two points.
The reason why the above case is  exceptional, is that $T$ maps the polynomial $x$  to $x^2$,
which has the same zeros as $x$. In general, such exceptional invariant sets 
only show up in the situation when there exists some $t$ such 
that $T(x-t)=c(x-t)^k$, see \cite{Beardon2000}.
\end{example}

\end{example}

\begin{remark}
There are at least two advantages in studying Hutchinson-invariant sets compared to the set-up of the present paper. 
The first one is that  the occuring types of fractal sets have already been
extensively studied which connects this topic to the existing classical complex dynamics,   
comp. e.g.~\cite{Barnsley1993,Falconer2004}.
The second advantage is that there exists a stable Monte--Carlo-type method for producing a good approximation  of $\hminvset{n}$, whenever the latter set is compact. Namely, 
\begin{enumerate}
 \item start with some $z_0 \in \bC$; 
 \item for $j=0,1,2,\dots$, pick randomly a root   of $T((x-z_j)^n)$ with equal probability, and denote it by $z_{j+1}$;
 \item plot $z_{j+1}$ and iterate step 2 until a picture emerges.
\end{enumerate}

Our experiments show that  about 100 iterations per final pixel gives a clear picture. This algorithm was used to create Figure~\ref{fig:julia}.
The set of points $z_j$ rapidly converge to the set $\hminvset{n}$,
and  the initial choice of $z_0$ statistically will not matter.
\end{remark}

Further information about Hutchinson-invariant sets can be found in a forth-coming paper \cite{He}. 

\medskip 
\noindent 
\textbf{Variation 3: Continuously Hutchinson-invariant sets.}
%


Given $T$ and $n$ as above,  consider 
\[
\psi(x,t,n) \coloneqq T((x-t)^n)/(x-t)^{n-k},
\]
where $k$ is the order of the operator $T$. Then $\psi(x,t,n)$
is a polynomial in $\bC[x,t,n]$. 
Given $n_0 \geq 0$, we say that
a set $S$ is \defin{continuously Hutchinson-invariant with parameter $\geq n_0$} 
if for every \emph{real} number $n\geq n_0$, we have that 
\[
 \psi(x,t_0,n) = 0 \qquad \text{ (considered as a polynomial in $\bC[x]$)}
\]
has all roots in $S$, whenever $t_0 \in S$. 
We denote by $\ctinvset{\geq n_0}$ the collection of all continuously 
Hutchinson-invariant with parameter $\geq n_0$  and by $\ctminvset{\geq n_0}$ 
the minimal non-empty closed such set $S$ (it if exists).
It is easy to verify that, for all integers $m \geq 1$, 
\[
\hminvset{m} \subseteq \ctminvset{\geq m} \subseteq \ctminvset{\geq 0}.
\]

Properties of the minimal continuously  Hutchinosn invariant set $\ctminvset{\geq n_0}$ seem 
to substantially depend on whether $n_0 =0$ or $n_0>0$: 
Namely,  the boundary of $\ctminvset{\geq 0}$ looks rectifiable,
while the boundary of $\ctminvset{\geq 1}$ seem to have a fractal (and non-rectifiable) character.
However, in contrast with Hutchinson-invariant sets which can be fractal, $\ctminvset{\geq n_0}$ 
always has a finite number of connected components. For operators $T$ of order $1$, continuously  Hutchinson invariant sets with positive  parameter have been 
studied in details in \cite{AHNST}. 




In general, it is unclear what the relation
between $\ctminvset{\geq n}$ and $\minvset{\geq n}$ is,
but for large $n$, we expect the inclusion $\ctminvset{\geq n} \subseteq \minvset{\geq n}$,
since extending the domain of $n$ from the set of large integers to the set of large real numbers does not seem to make a big difference.
Note that Theorem~\ref{thm:limit} and Proposition~\ref{prop:hutchinsonConvHull} suggest
that these sets coincide in the limit $n \to \infty$. 

The following proposition shows that as $n_0$
grows, the minimal continuously Hutchinson-invariant set 
converges to the zero locus of the leading coefficient $Q_k$ of $T$.
\begin{proposition}[Convergence to the zero locus of $Q_k$]\label{prop:hutchinsonConvHull}
Given a non-degenerate operator $T = \sum_{j=0}^k Q_j \diff{j}$,  $R>0$ and $\delta > 0$, 
then there exists  $n_0 = n_0(R,\delta)$ such that
for all $t \in \bC$, with $|t|<R$ we have that each root of 
\[
 T[ (x-t)^n ] =0
\]
different from $t$ lies at a distance at most $\delta$ from some root of $Q_k(x)$.

In particular, for any $\delta>0$,  there exists an $n_0 = n_0(\delta)$ 
such that  the $\delta$-neighborhood of the union of roots 
of $Q_k(x)$  is 
Hutchinson-invariant in degree  $n$, for all $n \geq n_0$.
The same holds for the continuously Hutchinson-invariant sets with parameter exceeding $n$.
\end{proposition}
\begin{proof}
Fix $R>0$ and $\delta>0$. A straightforward calculation shows that
\[
 \frac{\psi(x,t,n)}{n(n-1)\dotsm (n-k+1)} 
 = Q_k(x) + \sum_{j=1}^{k} \frac{Q_{k-j}(x)(x-t)^{j}}{(n-k+1)(n-k+2)\dotsm (n-k+j)}.
\]
Hence, the zeros $\psi(x,t,n)=0$ tend to the zeros $Q_k(x)$
as $n\to \infty$, provided that $|t| < R$.
Thus, for some $n_0 \coloneqq n_0(\delta)$, all roots of 
$\psi(x,t,n)=0$ lie at a distance at 
most $\delta$ from the fundamental polygon of $T$.
\end{proof}

\begin{figure}[!ht]
\centering
\includegraphics[width=0.25\textwidth]{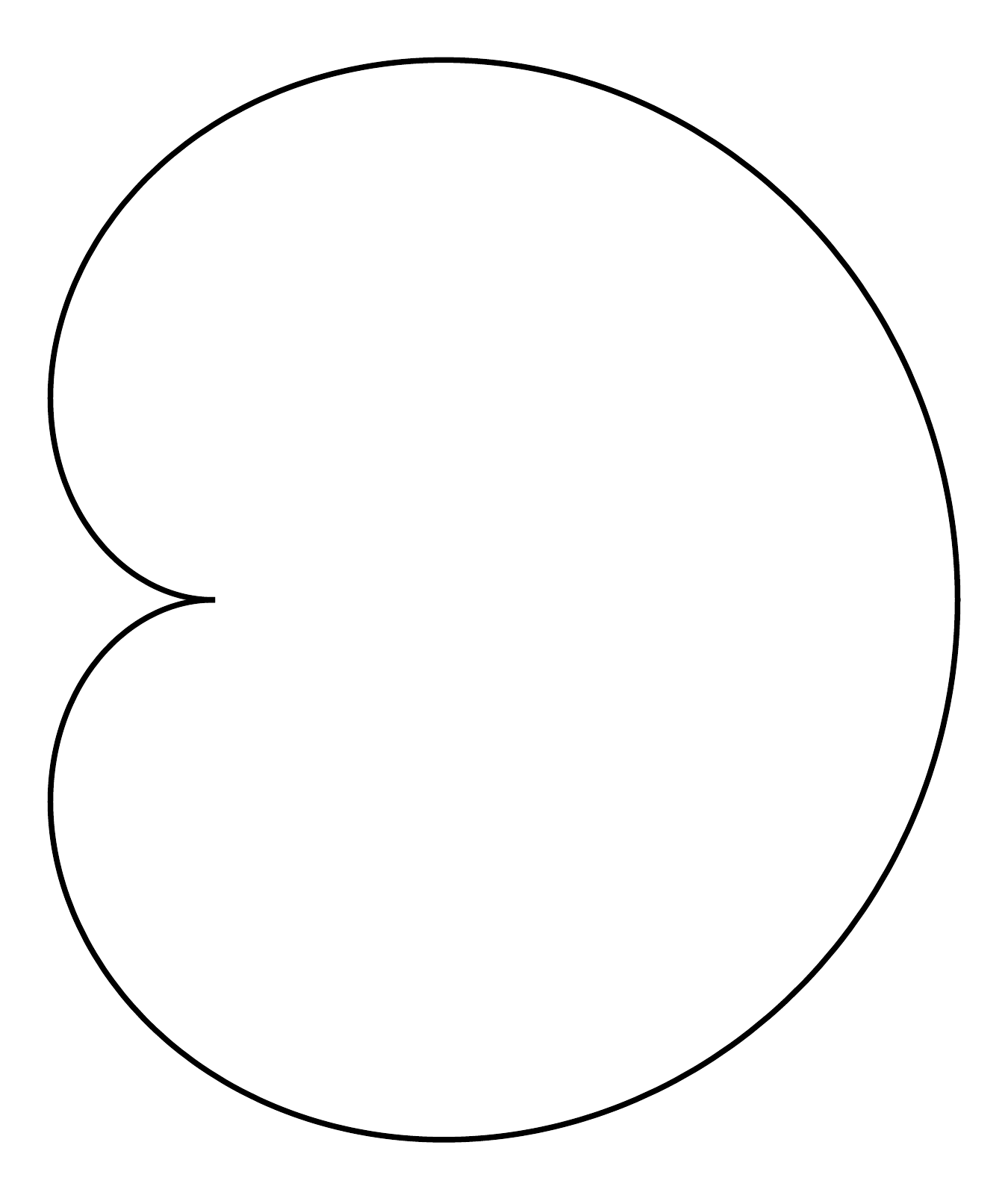}\hspace{1cm}
\includegraphics[width=0.4\textwidth]{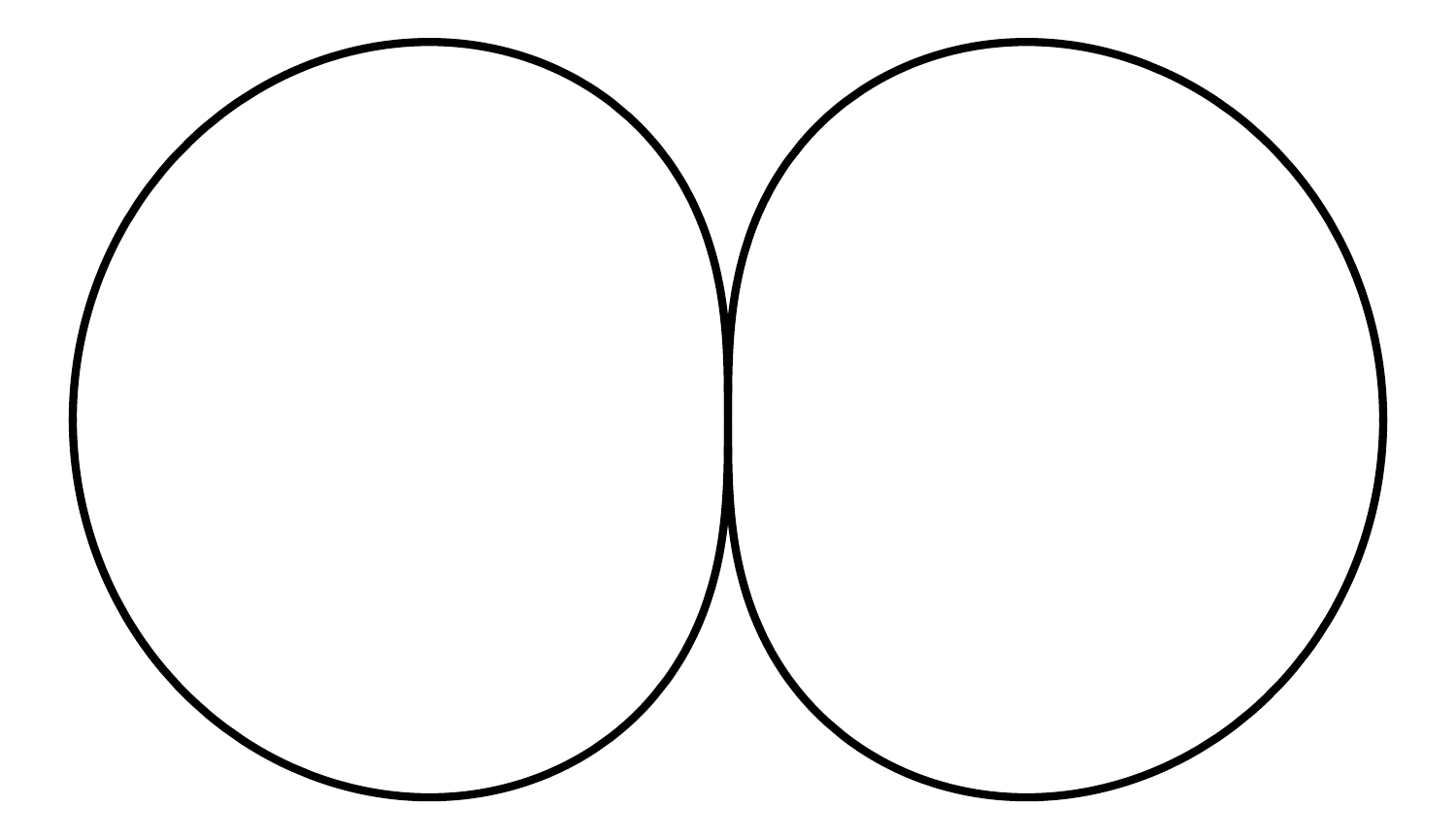}
\caption{The  boundaries of the  minimal continuously Hutchinson-invariant sets $\ctminvset{\geq 0}$ 
for the operators $T = z^2 \frac{d}{dz} + (z-1)$ (left), 
and $T = z^3 \frac{d}{dz} + (z+1)(z-1)$ (right). 
The first curve is parameterized by $r(\theta) = \frac{\sin{\theta}}{\theta}$
in polar coordinates, while the second is given by the equation $r^2(\theta) = \frac{\sin{2\theta}}{2\theta}$.
Proofs of these facts can be found in \cite{AHNST}.
}
\label{fig:cochleoid}
\end{figure}

\medskip 
\noindent 
\textbf{Variation 4: two-point continuously Hutchinson invariant sets.} 
 Our last variation of the notion of invariant sets is 
 inspired by the convexity property of invariant sets from $\invset{\ge n}$.  

Set $P(x) \coloneqq (x-t_1)^{n_1} (x-t_2)^{n_2}$
and consider 
\[
\phi(x,t_1,n_1,t_2,n_2) \coloneqq \frac{T(P)}{ (x-t_1)^{n_1-k} (x-t_2)^{n_2-k}},
\]
where $k$ is the order of the operator $T$. 
Again, $\phi(x,t_1,n_1,t_2,n_2)$ is a polynomial in $\bC[x,t_1,n_1,t_2,n_2]$. 
Given $n_0 \geq 0$, a set $S\subset \bC$ is called \defin{two-point continuously Hutchinson invariant with parameters $\geq n_0$} 
if for every pair of real number $n_1,n_2 \geq n_0$, we have that 
\[
 \phi(x,t_1,n_1,t_2,n_2) = 0 \qquad \text{(considered as a polynomial in $x$)}
\]
has all roots in $S$, whenever $t_1, t_2 \in S$. 
We denote by $\cttminvset{\geq n_0}$  the minimal under 
inclusion non-empty closed  set $S$ which is two-point Hutchinson invariant with parameters $\geq n_0$ (if it exists). 

Obviously, 
$\ctminvset{\geq n_0} \subseteq \cttminvset{\geq n_0}$.
Moreover, we can apply the same technique as in Theorem~\ref{th:generalGeN},
to show that two-point continuous invariant sets are \emph{convex}.

%

%


\begin{remark}

The linear operators which factor as in \eqref{eq:factorizable}
allow us to produce a large class of fractal sets associated with Hutchinson operators,
where each map is an affine contraction from $\bC$ to $\bC$.
These minimal invariant sets $\hminvset{n}$ are fractals, and therefore might be difficult to study.
It is highly plausible that 
continuously Hutchinson invariant set $\ctminvset{\geq 0}$
or its larger convex cousin $\cttminvset{\geq 0}$ have piecewise analytic boundary. For operators of order $1$, discussions of analyticity of the boundary of the former set can be found   in \cite{AHNST}. 
Remember that we have the set of inclusions
\[
 \hminvset{n} \subseteq \ctminvset{\geq 0} \subseteq \cttminvset{\geq 0},
\]
so a simple description of $\ctminvset{\geq 0}$ may provide some additional insight 
in the nature of $\hminvset{n}$. 

\end{remark}

\section{Some open problems}\label{sec:final}

Here we present  a very small sample of unsolved questions directly related to the results of this paper. 

\begin{enumerate}[label=\bfseries{\arabic*}.]
 \item 

The major open problem  is  whether it is possible 
to describe the boundary of $\minvset{\ge n}$ for non-degenerate or 
degenerate operators with non-defining Newton polygons and $Q_k$ different 
from a constant.  At the moment we only have some information what happens with $\minvset{\ge n}$ when $n\to \infty$. 
Already for no-degenerate operators of order $1$ this problem seems to be quite non-trivial, comp. \cite{AHNST}. 

\smallskip
\item
Another important  issue  is how $\minvset{\ge n}$  depend on the coefficients of operator $T$. 
It seems that even in the case when $T$ is non-degenerate and $n$ is such that  $\minvset{\ge n}$ is compact, 
it might loose compactness under small deformation of $T$ with the space of non-degenerate operators of the same order. 
Even for operators of order one the question is non-trivial. 
For example, consider the space of pairs of polynomials $(Q_1(x), Q_0(x))$ where $\deg Q_1(x)=k,\;\deg Q_0(x)=k-1$ and $T=Q_1(x)\frac{d}{dx}+Q_0(x)$. Fixing a positive integer $n$, 
is it possible to describe the space of such pairs $(Q_1(x), Q_0(x))$  for which  $\minvset{\ge n}$ is compact?

\smallskip
\item
Is it possible to characterize the invariant sets for Case B, i.e. operators with constant leading term and $\deg Q_{k-1} = 1$, see end of \S~\ref{sec:deg}.


\end{enumerate}

\bibliographystyle{amsalpha}
\bibliography{./theBibliography}

\end{document}